\documentclass[11pt]{article}

\usepackage[left=32mm,right=32mm,top=30mm,bottom=32mm]{geometry}
\usepackage{labelfig}
\usepackage{epsfig}
\usepackage{epstopdf}

\usepackage{caption}

\usepackage{xfrac}

\usepackage[percent]{overpic}

\usepackage{color}
\usepackage{amsthm,amsmath,amssymb}
\usepackage{booktabs}
\usepackage{mathpazo}
\usepackage{microtype}
\usepackage{overpic}
\usepackage{bm}
\usepackage{sectsty}
\usepackage[
	pdftitle={PDFTitle},
	pdfauthor={Hugo Parlier},
	ocgcolorlinks,
	linkcolor=linkred,
	citecolor=linkred,
	urlcolor=linkblue]
{hyperref}

\usepackage{mathtools}

\definecolor{linkred}{RGB}{199,21,133}
\definecolor{linkblue}{RGB}{16, 78, 139}

\usepackage[hang,flushmargin]{footmisc}
\usepackage{enumitem}
\usepackage{titlesec}
	\titlespacing{\section}{0pt}{12pt}{0pt}
	\titlespacing{\subsection}{0pt}{6pt}{0pt}
	
\titlelabel{\thetitle.\quad}

\makeatletter 

\long\def\@footnotetext#1{% 
\H@@footnotetext{% 
\ifHy@nesting 
\hyper@@anchor{\@currentHref}{#1}% 
\else 
\Hy@raisedlink{\hyper@@anchor{\@currentHref}{\relax}}#1% 
\fi 
}}

\def\@footnotemark{% 
\leavevmode 
\ifhmode\edef\@x@sf{\the\spacefactor}\nobreak\fi 
\H@refstepcounter{Hfootnote}% 
\hyper@makecurrent{Hfootnote}% 
\hyper@linkstart{link}{\@currentHref}% 
\@makefnmark 
\hyper@linkend 
\ifhmode\spacefactor\@x@sf\fi 
\relax 
}% 

\ifFN@multiplefootnote% 
\renewcommand*\@footnotemark{% 
\leavevmode 
\ifhmode 
\edef\@x@sf{\the\spacefactor}% 
\FN@mf@check 
\nobreak 
\fi 
\H@refstepcounter{Hfootnote}% 
\hyper@makecurrent{Hfootnote}% 
\hyper@linkstart{link}{\@currentHref}% 
\@makefnmark 
\hyper@linkend 
\ifFN@pp@towrite 
\FN@pp@writetemp 
\FN@pp@towritefalse 
\fi 
\FN@mf@prepare 
\ifhmode\spacefactor\@x@sf\fi 
\relax% 
}% 
\fi 

\makeatother 

\newtheorem{thm}{Theorem}[section]

\newtheorem{coro}[thm]{Corollary}
\newtheorem{lem}[thm]{Lemma}
\newtheorem{prop}[thm]{Proposition}
\newtheorem{conj}[thm]{Conjecture}

\theoremstyle{definition}

\theoremstyle{remark}

\newtheorem{remk}[thm]{Remark}

\renewcommand{\phi}{\varphi}

\newcommand{\sys}{{\rm sys}}

\newcommand{\tgamma}{\widetilde{\gamma}}

\newcommand{\talpha}{\widetilde{\alpha}}

\newcommand{\tP}{\widetilde{P}}

\newcommand{\tomega}{\widetilde{\omega}}
\newcommand{\tN}{\widetilde{N}}

\newcommand{\tX}{\widetilde{X}}
\newcommand{\tx}{\widetilde{x}}

\newcommand{\tz}{\widetilde{z}}

\newcommand{\ty}{\widetilde{y}}

\newcommand{\oD}{\overline{D}}

\newcommand{\oB}{\overline{B}}

\newcommand{\dist}{{\rm dist}}

\newcommand{\area}{{\rm area}}

\newcommand{\oq}{\overline{q}}

\newcommand{\oc}{\overline{c}}

\newcommand{\ov}{\overline{v}}

\newcommand{\Hcal}{ {\mathcal H}}

\newcommand{\Vcal}{ {\mathcal V}}

\newcommand{\oK}{\overline{K}}
\newcommand{\op}{\overline{p}}
\newcommand{\oalpha}{\overline{\alpha}}
\newcommand{\obeta}{\overline{\beta}}

\newcommand{\odelta}{\overline{\delta}}

\newcommand{\Cbb}{{\mathbb C}}

\newcommand{\Rbb}{{\mathbb R}}
\newcommand{\Zbb}{{\mathbb Z}}

\newcommand{\dev}{{\rm dev}}

\newcommand{\be}{ \begin{equation} }
\newcommand{\ee}{ \end{equation} }

\newcommand{\hol}{{\rm hol}}

\sectionfont{\large \bfseries}
\subsectionfont{\normalsize}

\long\def\symbolfootnote[#1]#2{\begingroup%
\def\thefootnote{\fnsymbol{footnote}}\footnote[#1]{#2}\endgroup}

\def\blfootnote{\xdef\@thefnmark{}\@footnotetext}

\date{\today}

\begin{document}

{\Large \bfseries 
The maximum number of systoles for genus two Riemann surfaces with abelian differentials}

{\large 
Chris Judge\symbolfootnote[2]{
Research partially supported by a Simons collaboration grant.} 
and Hugo Parlier\symbolfootnote[1]{
Research partially supported by Swiss National Science Foundation grant number PP00P2\textunderscore 153024. \vspace{.1cm} 
}
}

\vspace{0.5cm}

{\bf Abstract.} 
In this article, we provide bounds on systoles 
associated to a holomorphic $1$-form $\omega$ on a Riemann surface $X$. 
In particular, we show that if $X$ has genus two, then, up to homotopy, 
there are at most $10$ systolic loops on $(X, \omega)$
and, moreover, that this bound is realized by a unique translation surface 
up to homothety. For general genus $g$ and a holomorphic 1-form
$\omega$ with one zero, we provide the optimal 
upper bound, $6g-3$, on the number of homotopy classes of systoles.
If, in addition, $X$ is hyperelliptic, then we prove that the
optimal upper bound is $6g-5$.
\vspace{1cm}
\section{Introduction}

The {\em systolic length} of a length space $(X,d)$ 
is the infimum of the lengths of non-contractible loops in $X$.
If a non-contractible loop $\gamma$ achieves this infimum, then we will call
$\gamma$ a {\em systole}. The systolic length and systoles have received
a great deal of attention beginning with work of
Loewner who is credited \cite{Pu} with proving that among unit area Riemannian surfaces of genus one, the unit area hexagonal torus has the largest systolic length, $\sqrt{2/\sqrt{3}}$,
and is the unique such surface that achieves this value.

The hexagonal torus has another extremal property: Among all Riemannian
surfaces of genus one, it has the maximum number of distinct homotopy classes of systoles, three. With respect to this property, the hexagonal torus is not the unique
extremal among all genus one Riemannian surfaces, but it is the 
unique---up to homothety---extremal among
quotients of $\Cbb$ by lattices $\Lambda$ equipped with the metric $|dz|^2$.

The form $dz$ on $\Cbb/\Lambda$ is an example of a holomorphic 1-form
on a Riemann surface. More generally, given a holomorphic 1-form $\omega$ on a
Riemann surface $X$, one integrates $|\omega|$ over arcs
to obtain a length metric $d_{\omega}$ on $X$. 
On the complement of the zero set of $\omega$
the metric is locally Euclidean, and each zero of order $n$
is a conical singularity with angle $2 \pi \cdot (n+1)$.

The length space $(X, d_{\omega})$ determined
by $(X, \omega)$ is a basic object of study 
in the burgeoning field of Teichm\"uller dynamics.
See, for example, the recent surveys of \cite{Forni-Matheus-survey}
and \cite{Wright-survey}.

In this paper we prove the following.

\begin{thm} \label{thm:main}
Let $\omega$ be a holomorphic 1-form on a closed Riemann surface $X$ of genus two.
The number of distinct homotopy classes of systolic minimizers on $(X, d_{\omega})$
is at most 10. Moreover, up to homothety, there is a unique metric space of the form
$(X, d_{\omega})$ for which there exist exactly 10 distinct homotopy classes of systoles.
\end{thm}

In other words, among the unit area surfaces $(X, d_{\omega})$ of genus two,
there exists a unique surface $(X_{10}, d_{\omega_{10}})$ that achieves the maximum 
number of systolic homotopy classes. The surface obtained by multiplying the unit area 
metric $d_{\omega_{10}}$ by $\sqrt{4 \sqrt{3}}$ is described in Figure \ref{X10Figure}. 
The surface $(X_{10}, d_{\omega_{10}})$ has two conical singularities each of angle $4\pi$ 
corresponding to the vertices of the polygon pictured in Figure \ref{X10Figure}. 
In other words, the $1$-form $\omega_{10}$ has simple zeros corresponding to these vertices. 
Four of the ten systolic homotopy classes consist of geodesics that lie in an embedded 
Euclidean cylinders. Each of the other six systolic homotopy classes has a unique geodesic 
representative that necessarily passes through one of the two zeros of $\omega_{10}$. 
It is interesting to note that some
of the latter systoles intersect twice. Both intersections necessarily occur at 
zeros of $\omega_{10}$. Indeed, if two curves intersect twice and one of the 
intersection points is a smooth point of the Riemannian metric, then a standard 
perturbation argument produces a curve of shorter length.

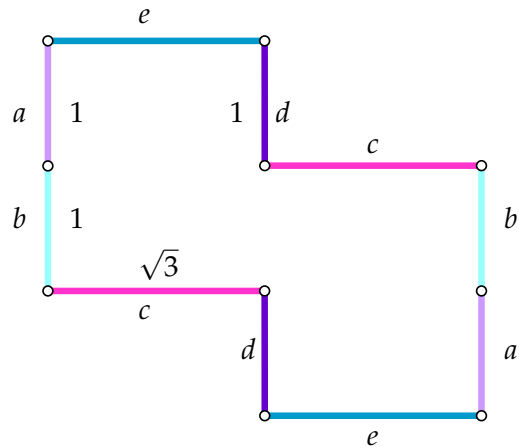
\begin{figure}[h]
%\ShowGrid
%{\color{linkred}
\leavevmode \SetLabels
\L(.39*.38) $\sqrt{3}$\\%
\L(.39*.27) $c$\\%
\L(.39*1.02) $e$\\%
\L(.33*.77) $1$\\%
\L(.28*.77) $a$\\%
\L(.33*.50) $1$\\%
\L(.28*.50) $b$\\%
\L(.47*.77) $1$\\%
\L(.71*.17) $a$\\%
\L(.71*.50) $b$\\%
\L(.59*.69) $c$\\%
\L(.59*-.05) $e$\\%
\L(.48*.17) $d$\\%
\L(.51*.77) $d$\\%
%\L(.398*.27) $r_0$\\
%\L(.595*.97) $\delta$\\
%\L(.49*-0.08) $\delta_{-}$\\
\endSetLabels
\begin{center}
\AffixLabels{\centerline{\epsfig{file =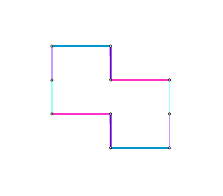,width=6cm,angle=0}}}
\vspace{-24pt}
\end{center}
\caption{A pair $(X, \omega)$ that has ten systoles: By identifying
parallel sides with the same letters, we obtain a Riemann surface $X$. The one form $dz$
in the plane defines a holomorphic 1-form on $X$.}\label{X10Figure}
\end{figure}

Perhaps surprisingly, $(X_{10}, d_{\omega_{10}})$ does not maximize the systolic length among all unit area, genus two surfaces of the form $(X, d_{\omega})$. To discuss this, it will be convenient to introduce the {\em systolic ratio}: the square of the systolic length divided by the area of the surface. A surface maximizes the systolic length among unit area surfaces if and only if it maximizes systolic ratio among all surfaces.

A genus two surface $(X, d_{\omega})$ that has ten systoles has systolic ratio equal to
$1/\sqrt{3} = .57735\ldots$. On the other hand, the surface described in Figure \ref{fig:maxratio} has systolic ratio equal to 
\begin{equation} \label{const:max}
 \frac{2 \cdot \left(\sqrt{13} -3 \right)^2 }{
 \sqrt{3} \cdot (1 -\frac{3}{4}(\sqrt{13}-3)^2)}~ =~ .58404\ldots
\end{equation}
We believe that this surface has maximal systolic ratio.
\begin{conj}
The supremum of the systolic ratio over surfaces $(X, d_{\omega})$ of genus two equals the constant in (\ref{const:max}). Moreover, up to homothety, the surface described in Figure \ref{fig:maxratio} is the unique surface that achieves this systolic ratio.
\end{conj}

\begin{figure}[h]
%\ShowGrid
%{\color{linkred}
\leavevmode \SetLabels
%\L(.39*.37) $\sqrt{3}$\\%
\L(.344*.45) $2$\\%
\L(.25*.47) $1$\\%
\L(.213*.518) $a$\\%
\L(.776*.49) $a$\\%
\L(.235*.2) $b$\\%
\L(.753*.78) $b$\\%
\L(.36*-0.02) $c$\\%
\L(.63*1.00) $c$\\%
\L(.36*.62) $e$\\%
\L(.63*.355) $e$\\%
\L(.48*.18) $d$\\%
\L(.51*.81) $d$\\%
%\L(.398*.27) $r_0$\\
%\L(.595*.97) $\delta$\\
%\L(.49*-0.08) $\delta_{-}$\\
\endSetLabels
\begin{center}
\AffixLabels{\centerline{\epsfig{file =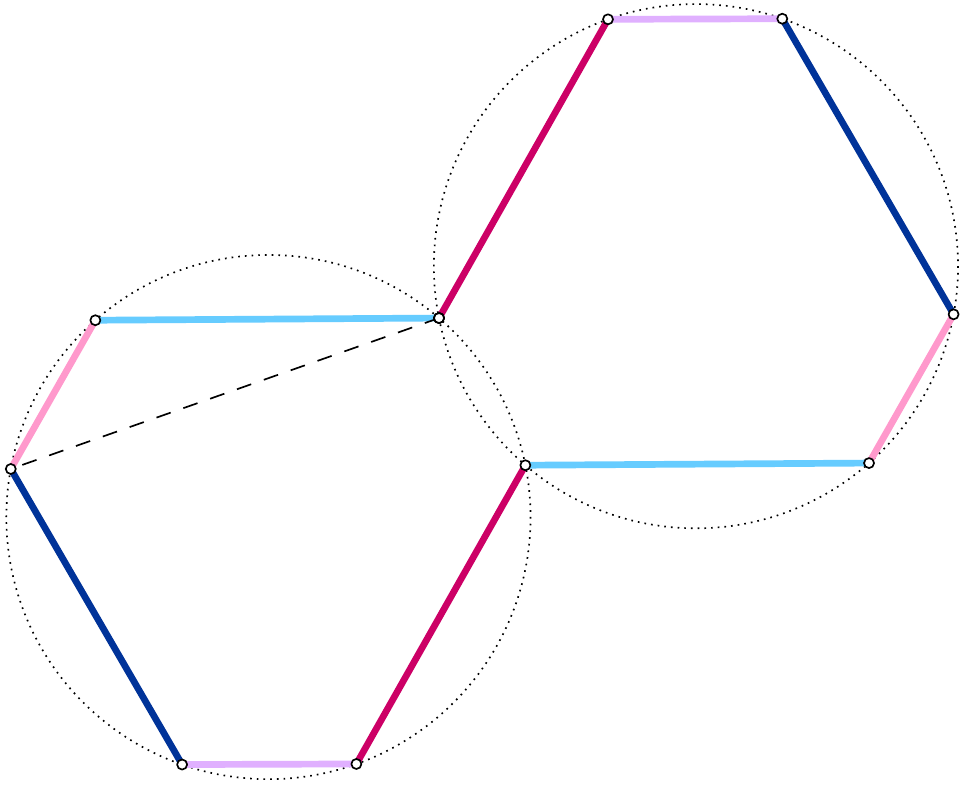,width=9cm,angle=0}}}
\vspace{-30pt}
\end{center}
\caption{A surface $(X,d_{\omega})$
whose systolic ratio equals the constant in (\ref{const:max}). The surface
is obtained from gluing parallel sides of two isometric cyclic hexagons in $\Cbb$.
Each hexagon has a rotational symmetry of order 3. The 1-form $\omega$ corresponds to $dz$ in the plane.}\label{fig:maxratio}
\end{figure}

By the Riemann-Roch theorem, the total number of zeros, including multiplicities,
of a holomorphic 1-form on a Riemann surface of genus $g$ equals $2g-2$.
In particular, a 1-form $\omega$ on a genus two Riemann surface $X$
consists of either two simple zeros or one double zero. Thus, we have a partition
of the moduli space of pairs $(X, \omega)$ into the stratum, $\Hcal(1,1)$, of those 
for which $d_{\omega}$ has two conical singularities of angle $4\pi$ and the 
complementary stratum, $\Hcal(2)$, those for which $d_{\omega}$ has a single conical 
singularity of angle $6 \pi$.

In order to prove Theorem \ref{thm:main}, we study each stratum separately.
It turns out that the stratum $\Hcal(2)$ is considerably easier to analyse. 
Indeed, for $\Hcal(2)$ we are able to prove sharp bounds on both the systolic 
ratio and on the number of systolic homotopy classes. This is due to the fact that
if there is only one zero, then each homotopy class of systoles may be represented
by a single saddle connection. 

\begin{thm} \label{thm:oneconepoint-genus2}
If $(X, \omega) \in \Hcal(2)$, then $(X,d_{\omega})$ has at most 7 homotopy classes of systoles, 
and the systolic ratio of $(X,d_{\omega})$ is at most $2/(3\sqrt{3}) = .3849\ldots$
Furthermore, either inequality is an equality if and only if $(X, d_{\omega})$ is tiled by an
equilateral triangle.\footnote{
A surface $(X, d_{\omega})$ is {\em tiled by an equilateral triangle $T$} 
if there exists a triangulation of $X$ such that each triangle is isometric to $T$ 
and each vertex is a zero of $\omega$.}
\end{thm}

The unique surface that attains both optimal bounds is illustrated 
in Figure \ref{fig:genus2equilateral}.

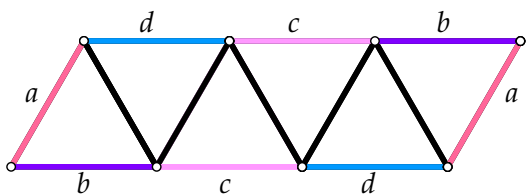
\begin{figure}[h]
%\ShowGrid
%{\color{linkred}
\leavevmode \SetLabels
%\L(.39*.37) $\sqrt{3}$\\%
\L(.29*.51) $a$\\%
\L(.71*.51) $a$\\%
\L(.335*-0.12) $b$\\%
\L(.65*1.00) $b$\\%
\L(.46*-0.12) $c$\\%
\L(.52*1.00) $c$\\%
\L(.584*-.12) $d$\\%
\L(.39*1.00) $d$\\%
%\L(.398*.27) $r_0$\\
%\L(.595*.97) $\delta$\\
%\L(.49*-0.08) $\delta_{-}$\\
\endSetLabels
\begin{center}
\AffixLabels{\centerline{\epsfig{file =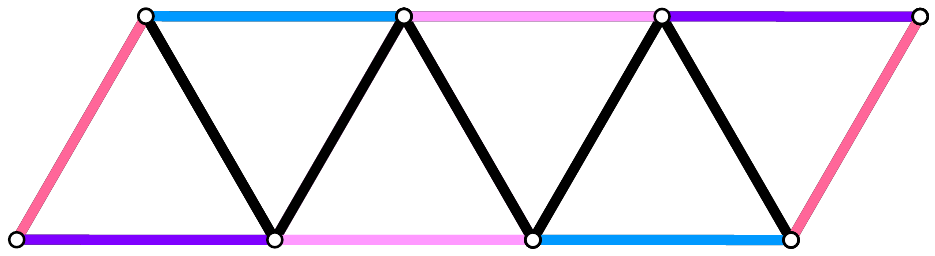,width=7cm,angle=0}}}
\caption{The genus two surface $(X, d_{\omega})$ that achieves the optimal bounds in $\Hcal(2)$. 
\label{fig:genus2equilateral} }
\vspace{-30pt}
\end{center}

\end{figure}

To prove the optimal systolic bounds for a holomorphic 1-form with one zero, 
we adapt the argument that Fejes T\'oth used to prove that a hexagonal lattice 
gives the optimal disc packing of the plane \cite{Toth}. This method of proof
extends to higher genus surfaces equipped with holomorphic one forms that have exactly 
one zero. 
\begin{thm} \label{thm:oneconepoint}
If $(X, \omega) \in \Hcal(2g-2)$, then the systolic ratio of $(X,d_{\omega})$ 
is at most $\frac{4}{(4g-2)\sqrt{3}}$. Equality is achieved if and only 
if the surface is tiled by equilateral triangles whose vertices lie at the zero of $\omega$.
\end{thm}

Theorem \ref{thm:oneconepoint} has been independently observed by Boissy and Geninska \cite{Boissy-Geninska}.

As indicated above, when $\omega$ has only one zero, each systole is homotopic
to a saddle connection of the same length. Smillie and Weiss \cite{Smillie-Weiss}
provided an upper bound on the length $\ell_0$ of the shortest saddle connection 
for surfaces $(X, d_{\omega})$ of genus $g$ and area 1. In particular, they showed
that $\ell_0 \leq \sqrt{1/\pi \cdot (2g-2+n)}$ where $n$ is the number of zeros of $\omega$.

We also identify optimal bounds for the number of homotopy classes 
of systoles of surfaces in $\Hcal(2g-2)$, and show that the optimal bounds are not 
attained by hyperelliptic surfaces in these strata. A condensed version of these results 
is the following (see Proposition \ref{thm:2g-2} and Theorem \ref{thm:hyperelliptic-2g-2}):

\begin{thm} \label{thm:condense}
If $\omega$ is a holomorphic 1-form on $X$ that has exactly one zero,
then $(X, d_{\omega})$ has at most $6g-3$ homotopy classes of systoles. 
If in addition $\omega$ is hyperelliptic, then $(X, d_{\omega})$ 
has at most $6g-5$ homotopy classes of systoles. Both bounds are realized. 
\end{thm}

The bulk of the present paper verifies Theorem \ref{thm:main} 
for the stratum  $\Hcal(1,1)$.
The proof begins in \S 4 where we show that each nonseparating systole is 
homotopic to a systole that passes through exactly two Weierstrass points. 
Such a systole is divided by the Weierstrass points into two geodesic arcs
that are interchanged by the hyperelliptic involution $\tau$.  
We regard each such `systolic Weierstrass arc' as an arc on 
$X / \langle \tau \rangle$ that joins the two corresponding angle $\pi$ cone points.
If a Weierstrass arc misses the angle $4\pi$ cone point on $X / \langle \tau \rangle$
that corresponds to the zeros of $\omega$, then we will call it `direct'.
Otherwise, the arc will be called `indirect'.  
In \S 5 we show that for each angle $\pi$ cone point $c$ there are
at most two direct systolic Weierstrass arcs that have an endpoint at $c$, 
and hence there are at most six homotopy classes of systoles that correspond
to direct Weierstrass arcs. 

The angle $4\pi$ cone point $c^*$ divides each indirect systolic Weierstrass arc into two subarcs
that we call `prongs'. Observe that if some prong has length $\ell \leq \sys(X)/4$,
then each of the other prongs emanating from $c^*$ has length $\sys(X)/2 -\ell$. 
In \S 6 we show that if all of the prongs have the same length---necessarily $\sys(X)/4$---then 
there are at most four prongs, and if there is a `short' prong of length $\ell < \sys(X)/4$,
then there are at most five prongs. In the former case, we obtain 
at most six indirect systolic Weierstrass arcs and in the latter case, we obtain
at most five indirect systolic Weierstrass arcs.\footnote{Note that a particular prong can lie in more 
than one systolic Weierstrass arc.} 

In \S 7 we show that there is at most one systole
that is a separating curve. Moreover, we show that if the surface has a  
systole which is a separating curve, then the surface has either no prongs 
or exactly two prongs of equal length. It follows that a surface with a separating systole 
has at most eight homotopy classes of separating systoles. 

In \S 8, we show that if there are exactly four prongs of equal length, 
then the surface has at most ten homotopy classes of systoles, and if there are ten,
then the surface is homothetic to the surface described in Figure \ref{X10Figure}. 
In \S 9, we show that if one of the prongs is shorter than the others, then 
there are at most none homotopy classes of systoles. This finishes the proof
of Theorem \ref{thm:main} in the case of surfaces from the stratum $\Hcal(1,1)$.

Although the questions that we address in this paper regarding systoles have not been systematically 
studied previously in the context of translation surfaces, 
they have been studied in the context of hyperbolic and general Riemannian surfaces. 
As hinted at above, smooth surfaces have systoles that intersect at most once, and from 
this one can deduce that there are at most $12$ homotopy classes of systole in genus two 
(see for instance \cite{Malestein-Rivin-Theran}). 
This bound is sharp. Indeed, among hyperbolic surfaces of genus two, there is a unique surface, 
called the Bolza surface, with exactly 12 systoles. It can be obtained by gluing opposite edges 
of a regular hyperbolic octagon with all angles $\frac{\pi}{4}$. This same surface is also optimal 
(again among hyperbolic surfaces) for systolic ratio, a result of Jenni \cite{Jenni}. 
There are bounds on these quantities in higher genus, but these bounds are not optimal.  
Interestingly, Katz and Sabourau \cite{Katz-Sabourau} showed that among $\rm{CAT}(0)$ genus 
two surfaces, the optimal surface is an explicit flat surface with cone point singularities, 
conformally equivalent to the Bolza surface. This singular surface can not be optimal among 
all Riemannian surfaces however, as by a result of Sabourau, the optimal surface in genus two 
necessarily has a region with positive curvature \cite{Sabourau}. The optimal systolic ratio 
among all Riemannian surfaces is still not known. 

{\bf Acknowledgements.}
We are grateful to the referee for a careful reading of the paper and valuable comments. 
We thank Marston Condor for the examples in Remark \ref{rmk:condor}. We thank Carlos Matheus 
Santos and Gabriela Weitze-Schmith\"usen for kindly pointing out some mistakes 
as well as some missing references in earlier versions. H. P. acknowledges support from U.S. 
National Science Foundation grants DMS 1107452, 
1107263, 1107367 RNMS: Geometric structures And Representation varieties (the GEAR Network).
C. J. acknowleges support from the Simons Foundation. 

%%%%%%%%%%%%%%%%%%%%%%%%%

\section{Facts concerning the geometry of $(X, d_{\omega})$} \label{section:lfacts}

We collect here some relevant facts about the geometry of the surface $(X, d_{\omega})$
sometimes called a `translation surface'. Much of this material can be found in, for example,
\cite{Masur-Smillie}, \cite{Gutkin-Judge}, and \cite{Broughton-Judge}.

\subsection{Integrating the 1-form}

By integrating the holomorphic $1$-form $\omega$ along a piecewise differentiable
path $\alpha: [a,b] \rightarrow X$, we obtain a path in
$\oalpha: [a,b] \to \Cbb$ defined by
\begin{equation} \label{defn:path-holonomy}
\oalpha(t)~ =~ \int_{\alpha|_{[a,t]}} \omega.
\end{equation}
Since $\omega$ is closed, if two paths $\alpha$, $\beta$ in $X$
are homotopic rel endpoints, then $\oalpha$ and $\obeta$ are homotopic rel endpoints.
Thus, if $U \subset X$ is simply connected neighborhood of a point $x$, then
\begin{equation} \label{defn:chart}
 \mu_{x,U}(y)~ :=~ \int_{\alpha_y} \omega
\end{equation}
is independent of the path $\alpha_y$ joining $x$ to $y$.
Note that $\mu_{x,U}$ is a holomorphic map from $U$ into $\Cbb$.
If $x$ is not a zero of $\omega$, then it follows from the inverse function
theorem that there exists a neighborhood $U$ so that $\mu_{x,U}$ is
a biholomorphism onto its image.

\subsection{The metric}

The norm, $|\omega|,$ of $\omega$ defines an arc length element on $X$.
We will let $\ell_{\omega}(\alpha)$ 
denote the length of a path on $X$, and we will let $d_{\omega}$ denote the metric
obtained by taking the infimum of lengths of paths joining two points.

If $x$ is not a zero of $\omega$ and $U$ is a simply connected neighborhood of $x$,
then $\mu_{x,U}$ is a local isometry from $U$ into $\Cbb$ equipped with its usual Euclidean 
metric $|dz|^2$. If, in addition, $U$ is star convex at $x$, then $\mu_{x,U}$ is an isometry 
onto its image.

If $x$ is a zero of $\omega$ of order $k$, then there exists a neighborhood $V$ of $x$
and a chart $\nu: V \to \Cbb$ such that $\omega= (k+1) \cdot \nu^*(z^{k} dz)) = \nu^*(d (z^{k+1}))$ 
and $\nu(x)=0$. If $V$ is sufficently small, the map $\nu$ is an isometry from $(V, d_{\omega})$ 
to $(\nu(V), d_{d (z^{k+1})})$. In turn, the map $z \mapsto z^{k+1}$ is a local isometry from 
$(\nu(V) - \{0\}, d_{d (z^{k+1})})$ to a neighborhood of the origin with the Euclidean metric 
$|dz|^2$. Since the branched covering $z \mapsto z^{k+1}$ has degree $k+1$, the arc length 
of the boundary of an $\epsilon$-neighborhood of $x$ is $2\pi (k+1) \cdot \epsilon$. 
Therefore, we refer to $x$ as a {\em cone point} of angle $2 \pi (k+1)$, and the set of zeros, 
denoted $Z_\omega$, will be regarded as the set of cone points of $(X, d_{\omega})$. 

\subsection{Universal cover, developing map and holonomy}
Let $p: \tX \to X$ be the universal covering map, and let $\tomega = p^*(\omega)$.
If we let $d_{\tomega}$ be the associated metric on $\tX$, then $p$ is a local
isometry from $(\tX, d_{\tomega})$ onto $(X, d_{\omega})$. Since $\tX$ is simply connected, 
we may fix $\tx_0 \in \tX$ and integrate $\tomega$ as in (\ref{defn:chart}) to obtain a 
map $\dev: \tX \to \Cbb$ called the {\em developing map}. The restriction of $\dev$ to 
$\tX - Z_{\tomega}$ is a local biholomorphism and a local isometry. Each zero of $\tomega$ 
is a branch point whose degree equals the order of the zero. If $C$ is the closure of a 
convex subset of $\tX - Z_{\tomega}$, then the restriction of $\dev$ to $C$ is injective.
 
Let $x_0= p(\tx_0)$, and consider loops $\alpha$ in $x$ based at $x_0$. The assignment 
$\alpha \to \oalpha$ defines a homomorphism, $\hol$, from $\pi_1(X, x_0)$ to the additive group $\Cbb$. 
Moreover, for each $[\alpha] \in \pi_1(X, x_0)$ and $\tx \in \tX$ we have
\begin{equation} \label{dev-equivariance}
 \dev( [\alpha] \cdot \tx)~ =~ \dev(\tx)~ +~ \hol([\alpha])
\end{equation}
where $\alpha \cdot \tx$ denotes action by covering transformations. 
See, for example, \cite{Gutkin-Judge}.

\subsection{Geodesics}
If a geodesic $\gamma$ on $(X, d_{\omega})$ passes through a zero of $\omega$, then
$\gamma$ will be called {\em indirect} and otherwise {\em direct}. If $\gamma$ is a direct simple geodesic loop, then, since $Z_{\omega}$ is finite, for sufficiently small $\epsilon>0$, the $\epsilon$-tubular neighborhood, $N$, of $\gamma$ is disjoint from $Z_{\omega}$. Each lift $\tN \subset \tX$ of $N$ is convex and hence the restriction of the developing map to $\tN$ is an isometry onto $\dev(\tN)$. Since $\tN$ is stabilized by the cyclic subgroup $\langle \gamma \rangle$ of the deck group generated by $\gamma$, it follows from (\ref{dev-equivariance}) that $\dev(\tN)$ is the convex hull of two parallel lines, and, moreover, the map $\dev$ determines an isometry from $N$ to 
$\dev(\tN)/ \langle \hol( \gamma) \rangle$. In particular, $N$ is isometric to a Euclidean cylinder $[0,w] \times \Rbb/\ell \Zbb$ where $\ell= |\hol(\gamma)|$ and $w$ is the distance between the parallel lines. If $Z_{\omega} \neq 0$, then the union of all Euclidean cylinders embedded in $X-Z_{\omega}$ that contain $\gamma$ is a cylinder called the {\em maximal cylinder} associated to $\gamma$. Each component of the frontier of a maximal cylinder consists of finitely many indirect geodesics.

\begin{prop} \label{prop:passes-thru-zero}
If $\omega$ has at least one zero, then
each homotopy class of loops is represented by a
geodesic loop that passes through a zero of $\omega$. 
\end{prop} 

\begin{proof}
Since $X$ is compact, a homotopy class of simple loops has a geodesic representative $\gamma$.
If $\gamma$ does not pass through a zero, then $\gamma$ lies in a maximal cylinder.
The boundary of the maximal cylinder contains a geodesic representative that
passes through a zero. 
\end{proof}

\begin{prop} \label{prop:maxcyl}
If two direct 
simple geodesic loops are homotopic, then they lie in the closure of the same maximal cylinder.
\end{prop}
\begin{proof}
Because the angle at each cone point $\tz \in Z_{\tomega}$
is greater than $2 \pi$, the length space $(\tX, d_{\tomega})$ is CAT(0). If two geodesic
loops $\gamma$ and $\gamma'$
are homotopic, then they have lifts that are asymptotic in $(\tX, d_{\tomega})$.
By the flat strip theorem \cite{BridsonHaefliger}, the convex hull of the two lifts
is isometric to a strip $[0,w] \times \Rbb$. Thus, since each cone point has
angle larger than $2\pi$, the interior $I$ of the convex hull contains no cone points.
The developing map restricted to $I$ is an isometry
onto a strip in $\Cbb$, and, moreover, it induces an isometry from
$I/ \langle g \rangle$
to the cylinder $\dev(I)/\langle \hol(g) \rangle$ where $g$ is the deck transformation
associated to the common homotopy class of $\gamma$ and $\gamma'$. Since the lifts
are boundary components of $I$, the loops $\gamma$ and $\gamma'$ lie in the boundary
of the cylinder $\dev(I)/\langle \hol(g) \rangle$.
\end{proof}

%%%%%%%%%

\subsection{The Delaunay cell decomposition} \label{subsec:delaunay}

The Delaunay decomposition is well-known in the context of complete constant curvature
geometries. Thurston observed that the construction also applies to constant
curvature metrics with conical singularities \cite{Thr98}.

We will first describe the Delaunay decomposition of the universal cover $\tX$. 
Given $\tx \in \tX -Z_{\tomega}$, let $D_{\tx}$ be the largest open disk centered at
$\tx$ that does not intersect $Z_{\tomega}$. Since $D_{\tx}$ is convex, the restriction
of $\dev$ to the closure $\oD_{\tx}$ is an isometry onto a closed Euclidean disk in $\Cbb$.
Since $Z_{\tomega}$ is discrete, the intersection $Z_{\tomega} \cap \oD_{\tx}$ is finite.
Let $\Vcal$ be the set of $\tx \in \tX-Z_{\tomega}$ such that $Z_{\tomega} \cap \oD_{\tx}$ 
contains at least three points. Because three points determine a circle,
the set $\Vcal$ is discrete.

For each $\tx \in \Vcal$, let $P_{\tx}$ denote the convex hull of $Z_{\tomega} \cap \oD_{\tx}$. It is isometric to a convex polygon in the plane. Again, because three points determine a circle, if $\tx, \ty \in \Vcal$ and $\tx \neq \ty$, then the set $Z_{\tomega} \cap \oD_{\tx} \cap \oD_{\ty}$ consists of at most two points, and hence $P_{\tx} \cap P_{\ty}$ is either empty, a point, or a geodesic arc lying in both the boundary of $P_{\tx}$ and the boundary of $P_{\ty}$. The interior of $P_{\tx}$ is called a {\em Delaunay 2-cell} and
the boundary edges are called {\em Delaunay edges}. The vertex set of this decomposition of $\tX$ is the set of zeros of $\tomega$. 

The deck group of the universal covering map $p$ permutes the cells of the
Delaunay decomposition, and so we obtain a decomposition of $X$.
Note the restriction of $p$ to each $2$-cell $P$ is an isometry onto its image.
Indeed, if not then there exists a covering transformation $\gamma$,
a lift $\tP$ of $P$, and $\tx \in \tP$ such that $\gamma \cdot \tx \in \tP$. 
Since $\tP$ is convex, it follows that for some vertex $\tz \in \tP$, we would have 
$\gamma \cdot \tz \in \tP$. But $\gamma$ maps $Z_{\tomega}$ to itself. 

Our interest in the Delaunay decomposition stems from the following.

\begin{prop} \label{prop:shortDelaunayedge}
If $\alpha$ is a shortest non-null homotopic arc with both endpoints in $Z_{\omega}$,
then $\alpha$ is a Delaunay edge.
\end{prop}
\begin{proof}
Since the universal covering map $p$ preserves the length of arcs, it suffices to prove that the 
analogous
statement holds for the universal cover $\tX$. Because $\alpha$ is a shortest arc, if $m$ is
the midpoint of $\alpha$, then the largest disc $D$ centered at $m$ has diameter equal to $\ell(\alpha)$
and $\oD \cap Z_{\tomega}$ consists of exactly two points, the endpoints $z$ and $z'$ of $\alpha$. 
The circle $\dev(\partial D)$ belongs to the pencil of circles containing $\dev(z)$ and $\dev(z')$.
Since $X$ is compact, by varying over this pencil, we find a disk $D'$ so that $\oD' \cap Z_{\tomega}$ contains $z$, $z'$, and at least one other point. The center $c$ of $D'$ belongs to $\Vcal$ and $\alpha$ is a boundary edge of the polygon $P_{c}$.
\end{proof}

\begin{prop} \label{prop:Delaunay-count}
Let $\omega$ be a holomorphic 1-form on a closed surface of genus $g$.
If $\omega$ has $v$ zeros, then the Delaunay decomposition of $X$ has
at most $6g-6+3\cdot v$ edges and the number of 2-cells is at most $4g-4+2 \cdot v$.
Equality holds if and only if each $2$-cell is a triangle.
\end{prop}

\begin{proof}
By dividing the Delaunay 2-cells (convex polygons) into triangles,  we obtain a triangulation with 
$v$ vertices. By Euler's formula and the fact that there are 3 oriented edges for each triangle, 
we find that each triangulation has $6g-6 +3v$ edges and $4g-4+2 \cdot v$ triangles.
\end{proof}

%%%%%%%%%%%%%%%%%%%%%%%%%%%

\section{Systoles of 1-forms in $\Hcal(2g-2)$}

In this section, we consider holomorphic 1-forms with a single zero. In the first part of the section we give the optimal bound on the number of homotopy classes of systoles of such surfaces as well as the optimal bound for the hyperelliptic surfaces with one zero. 
In the second part, we provide the optimal estimate on the systolic ratio of such surfaces.

\subsection{Bounds on the number of systoles}

\begin{prop} \label{thm:2g-2}
If $\omega$ is a holomorphic 1-form on $X$ that has exactly one zero,
then $(X, d_{\omega})$ has at most $6g-3$ homotopy classes of systoles.
\end{prop}

\begin{proof}
By Proposition \ref{prop:passes-thru-zero}, each homotopy class of systoles
contains a representative that passes through the zero. Proposition
 \ref{prop:shortDelaunayedge} implies that each such systole is a 
Delaunay edge. By Proposition \ref{prop:Delaunay-count}, there are 
at most $6g-3$ Delaunay edges and hence at most $6g-3$ homotopy 
classes of systoles. 
\end{proof}

The bound in Proposition \ref{thm:2g-2} is sharp if the genus $g$ of $X$ 
is at least 3. For example, if $g=3,4,5$, then consider the 
surfaces described in Figures \ref{fig:Marston-example}, 
\ref{fig:Marston-example2},
and \ref{fig:Marston-example3}. 

\begin{figure}[h]
%\ShowGrid
%{\color{linkred}
\leavevmode \SetLabels
%\L(.39*.37) $\sqrt{3}$\\%
\L(.227*.51) $a$\\%
\L(.775*.51) $a$\\%
\L(.25*-0.1) $b$\\%
\L(.52*1.00) $b$\\%
\L(.36*-0.1) $c$\\%
\L(.73*1.00) $c$\\%
\L(.46*-.1) $d$\\%
\L(.63*1.00) $d$\\%
\L(.565*-.1) $e$\\%
\L(.43*1.00) $e$\\%
\L(.68*-.1) $f$\\%
\L(.32*1.00) $f$\\%
%\L(.398*.27) $r_0$\\
%\L(.595*.97) $\delta$\\
%\L(.49*-0.08) $\delta_{-}$\\
\endSetLabels
\begin{center}
\AffixLabels{\centerline{\epsfig{file=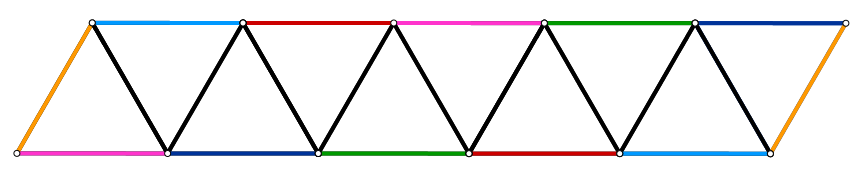,width=9cm,angle=0}}}
\caption{Genus 3 example that saturates bound in Proposition \ref{thm:2g-2}.
Glue the edges of the polygon according to the labels. Each edge is a systole, 
the 1-form $\omega$ has exactly one zero, and no two Delaunay
edges are homotopic. \label{fig:Marston-example} }
\vspace{-30pt}
\end{center}

\end{figure}
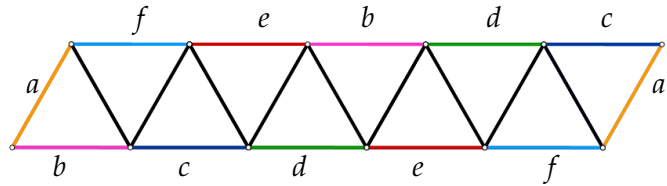

\begin{figure}[h]
%\ShowGrid
%{\color{linkred}
\leavevmode \SetLabels
%\L(.39*.37) $\sqrt{3}$\\%
\L(.13*.51) $a$\\%
\L(.865*.51) $a$\\%
\L(.165*-0.02) $b$\\%
\L(.82*.95) $b$\\%
\L(.27*-0.02) $c$\\%
\L(.62*.95) $c$\\%
\L(.36*-.02) $d$\\%
\L(.22*.95) $d$\\%
\L(.46*-.02) $e$\\%
\L(.715*.95) $e$\\%
\L(.56*-.02) $f$\\%
\L(.525*.95) $f$\\%
\L(.67*-0.02) $g$\\%
\L(.42*.95) $g$\\%
\L(.76*-0.02) $h$\\%
\L(.32*.95) $h$\\%
%\L(.398*.27) $r_0$\\
%\L(.595*.97) $\delta$\\
%\L(.49*-0.08) $\delta_{-}$\\
\endSetLabels
\begin{center}
\AffixLabels{\centerline{\epsfig{file=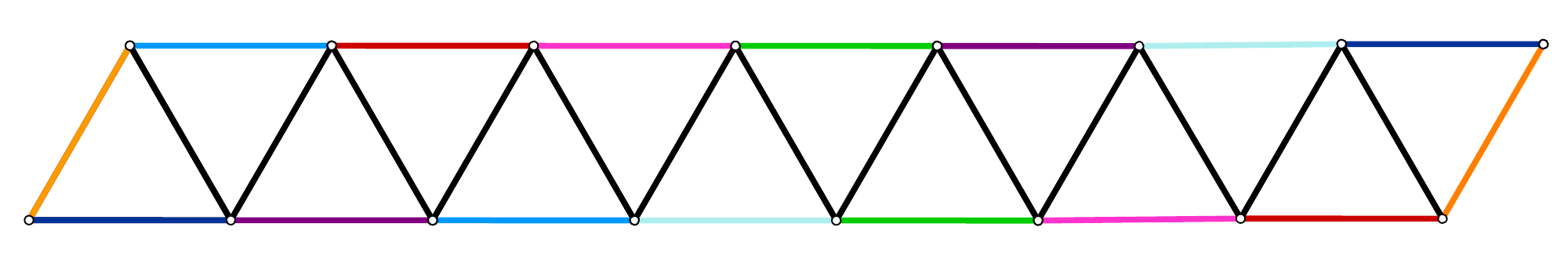,width=11.8cm,angle=0}}}
\caption{Genus 4 example that saturates bound in Proposition \ref{thm:2g-2}.
\label{fig:Marston-example2} }\vspace{-20pt}
\end{center}

\end{figure}

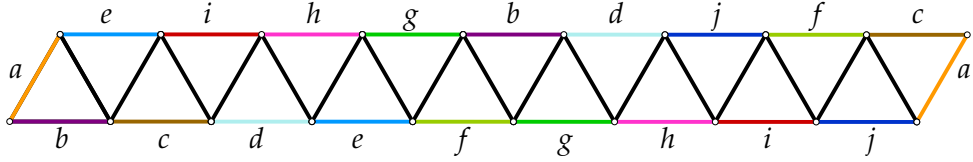
\begin{figure}[h]
%\ShowGrid
%{\color{linkred}
\leavevmode \SetLabels
%\L(.39*.37) $\sqrt{3}$\\%
\L(.08*.5) $a$\\%
\L(.91*.5) $a$\\%
\L(.12*-0.2) $b$\\%
\L(.515*1.05) $b$\\%
\L(.21*-0.2) $c$\\%
\L(.87*1.05) $c$\\%
\L(.29*-.2) $d$\\%
\L(.605*1.05) $d$\\%
\L(.38*-.2) $e$\\%
\L(.16*1.05) $e$\\%
\L(.47*-.2) $f$\\%
\L(.78*1.05) $f$\\%
\L(.56*-0.2) $g$\\%
\L(.425*1.05) $g$\\%
\L(.65*-0.2) $h$\\%
\L(.34*1.05) $h$\\%
\L(.74*-0.2) $i$\\%
\L(.25*1.05) $i$\\%
\L(.83*-0.2) $j$\\%
\L(.695*1.05) $j$\\%
%\L(.398*.27) $r_0$\\
%\L(.595*.97) $\delta$\\
%\L(.49*-0.08) $\delta_{-}$\\
\endSetLabels
\begin{center}
\AffixLabels{\centerline{\epsfig{file =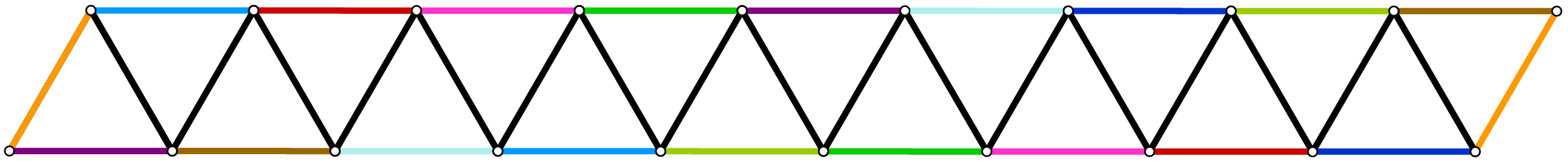,width=12.9cm,angle=0}}}
\caption{Genus 5 example that saturates bound in Proposition \ref{thm:2g-2}.
\label{fig:Marston-example3} }\vspace{-20pt}
\end{center}

\end{figure}

More generally, given a holomorphic 1-form $\omega_g$ on a surface $X_g$ of genus $g$ 
with one zero that achieves the bound $6g-3$, one can construct a holomorphic 
1-form $\omega_{g+3}$ with one zero on a surface $X_{g+3}$ of genus $g+3$ 
that achieves the bound $6(g+3)-3$. Indeed, remove the interior of a
Delaunay edge from $(X_g, d_{\omega_g})$ to obtain a surface $X'_g$ with 
`figure eight' boundary consisting of two segments $F_-$, $F_+$ each corresponding
to the Delaunay edge. Let $(Y_2,dz)$ be the genus two surface with 
two boundary components $G_-$, $G_+$ that is described in Figure \ref{fig:Marston-inductive}.
By gluing $F_{\pm}$ to $G_{\pm}$, we obtain the desired  $(X_{g+3},\omega_{g+3})$.

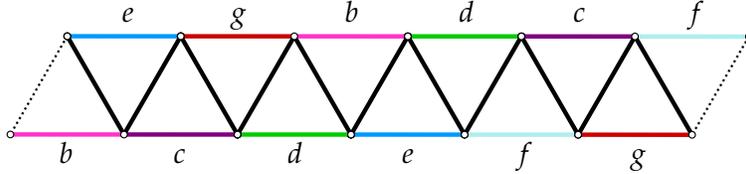
\begin{figure}[h]
%\ShowGrid
%{\color{linkred}
\leavevmode \SetLabels
%\L(.39*.37) $\sqrt{3}$\\%
\L(.22*-0.2) $b$\\%
\L(.47*1.05) $b$\\%
\L(.32*-0.2) $c$\\%
\L(.67*1.05) $c$\\%
\L(.42*-.2) $d$\\%
\L(.57*1.05) $d$\\%
\L(.52*-.2) $e$\\%
\L(.275*1.05) $e$\\%
\L(.62*-.2) $f$\\%
\L(.77*1.05) $f$\\%
\L(.72*-0.2) $g$\\%
\L(.37*1.05) $g$\\%
%\L(.398*.27) $r_0$\\
%\L(.595*.97) $\delta$\\
%\L(.49*-0.08) $\delta_{-}$\\
\endSetLabels
\begin{center}
\AffixLabels{\centerline{\epsfig{file =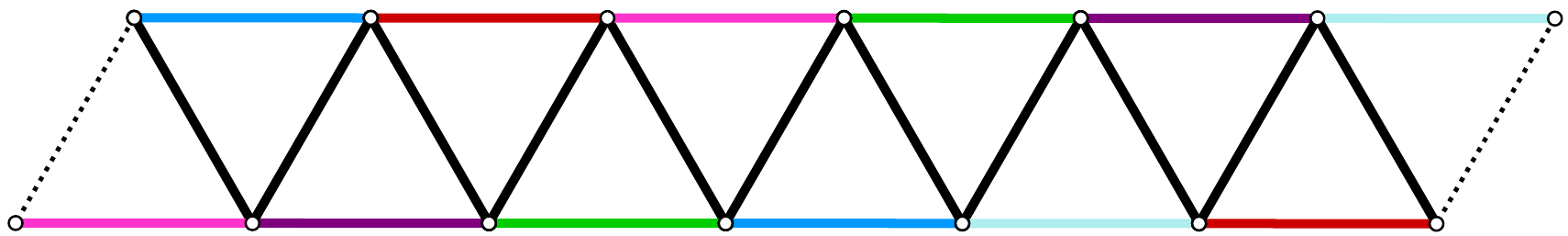,width=10cm,angle=0}}}
\caption{Glue solid edges of the polygon that have the same label to
obtain the Delaunay triangulation associated a holomorphic 1-form 
on a surface of genus two having two boundary components.
\label{fig:Marston-inductive} 
 }\vspace{-20pt}
\end{center}

\end{figure}
\begin{remk}\label{rmk:condor}
The problem of constructing surfaces that saturate the bound in 
Proposition \ref{thm:2g-2} is equivalent to the problem of constructing
two fixed-point free elements $\sigma$, $\tau$ in the 
symmetric group $S_{g-1}= {\rm Sym}(\{1, \ldots, 2g-1\})$ 
such that $\sigma \cdot \tau$ has no fixed points and the 
commutator $[\sigma, \tau]$ is a $(2g-1)$-cycle. Indeed, let $P_1, \ldots P_g$
be $2g-1$ disjoint copies of the convex hull of $\{0,1, e^{\pi i/3}, 1+ e^{\pi i/3}\}$.
Given $\sigma, \tau \in S_{2g-1}$, glue the left side of 
$P_i$ to the right side of $P_{\sigma(i)}$ and the 
top side of $P_i$ to the bottom side of $P_{\tau(i)}$
to obtain a surface with a holomorphic 1-form $\omega$. 
If $[\sigma, \tau]$ is an $n$-cycle, then it follows that 
$\omega$ has one zero, and if $\sigma$, $\tau$, and $\sigma \cdot \tau$ 
have no fixed points, then it follows that $(X, d_{\omega})$ has 
no cylinder with girth equal to the systole. Thus, by
Proposition \ref{prop:maxcyl}, no two systolic edges are homotopic.

Conversely, suppose that a holomorphic 1-form 
surface saturates the bound, then the necessarily equilateral 
Delaunay triangles can be paired to form parallelograms as above
that are glued according to permutations $\sigma$ and $\tau$.
One verifies that $\sigma$ and $\tau$ satisfy the desired
properties.

The surface constructed in Figure \ref{fig:Marston-example} corresponds to 
the pair $\sigma =(12345)$, $\tau=(15243)$, the surface constructed 
in Figure \ref{fig:Marston-example2} corresponds to the pair $\sigma=(1234567)$,
$\tau=(1364527)$, and the surface in Figure \ref{fig:Marston-example3} 
corresponds to $\sigma=(123456789)$,
$\tau=(146379285)$.
We thank Marston Condor for finding these examples for us. 
\end{remk}

%By Proposition \ref{prop:maxcyl},
%two edges are homotopic if and only if each is a boundary component
%of the same maximal cylinder. A Delaunay edge bounds a maximal cylinder
%if and only if the cylinder consists of two Delaunay triangles. 
%In particular, to produce a surface with $6g-3$ systoles, we need
%only produce a bicubic
If the genus of the surface is two, then one can show that the maximum number of homotopy classes of systoles is $7=6g-5$. More generally, the following is true. 
%genus two Riemann surface has a hyperelliptic involution, a conformal 
%map $\tau: X \to X$ such that the quotient is $X/ \langle \tau \rangle$
%is the Riemann sphere. Thus, the following shows that a 
%holomorphic 1-form with one zero on a surface genus 2 
%has at most 7 homomotopy classes of . 
\begin{thm}\label{thm:hyperelliptic-2g-2}
Let $\omega$ be a holomorphic 1-form on a surface with a 
hyperelliptic involution $\tau$. If $\omega$ has exactly one zero,
then $(X, d_{\omega})$ has at most $6g-5$ homotopy classes of systoles.
Moreover, $(X, d_{\omega})$ has exactly $6g-5$ homotopy classes of systoles
if and only if each Delaunay edge is a systole and there exist exactly 
four Delaunay 2-cells each of which have two edges that are preserved
by the hyperelliptic involution. 
\end{thm}

For each $g\geq2$, the bound given in Theorem \ref{thm:hyperelliptic-2g-2}
is achieved by, for example, the surface described in Figure \ref{fig:hyperelliptic-saturate}.
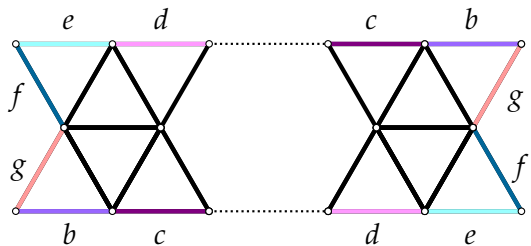
\begin{figure}[h]
%\ShowGrid
%{\color{linkred}
\leavevmode \SetLabels
%\L(.39*.37) $\sqrt{3}$\\%
\L(.32*-0.14) $b$\\%
\L(.672*1.04) $b$\\%
\L(.4*-0.14) $c$\\%
\L(.585*1.04) $c$\\%
\L(.585*-.14) $d$\\%
\L(.4*1.04) $d$\\%
\L(.672*-.14) $e$\\%
\L(.32*1.04) $e$\\%
\L(.714*.24) $f$\\%
\L(.273*.63) $f$\\%
\L(.71*.63) $g$\\%
\L(.275*.25) $g$\\%
%\L(.398*.27) $r_0$\\
%\L(.595*.97) $\delta$\\
%\L(.49*-0.08) $\delta_{-}$\\
\endSetLabels
\begin{center}
\AffixLabels{\centerline{\epsfig{file =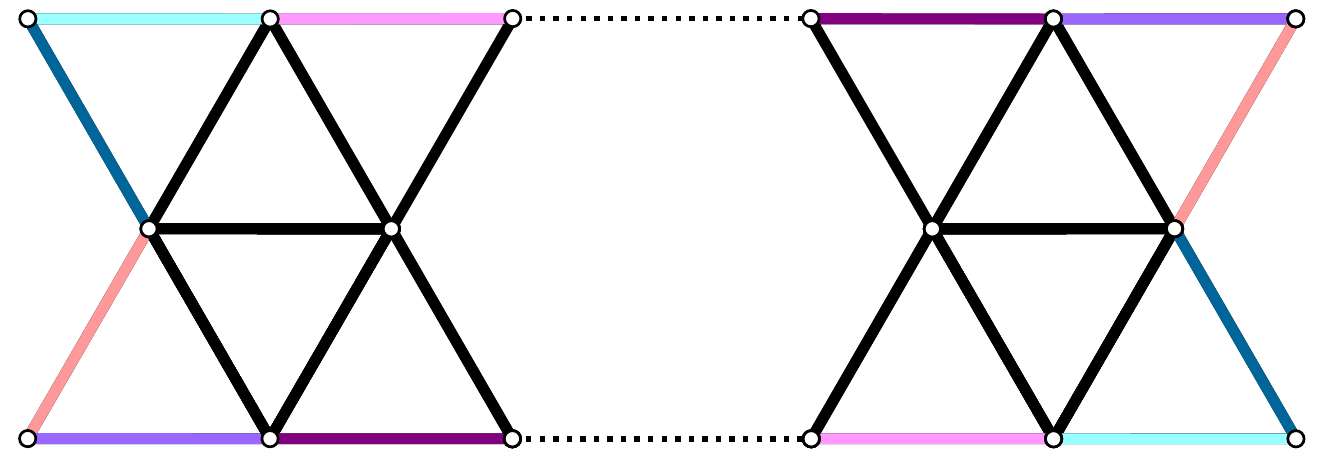,width=7cm,angle=0}}}
\caption{Glue edges of the polygon that have the same label to
obtain the Delaunay triangulation associated to a holomorphic 1-form 
on a surface of genus g. The surface is hyperelliptic, 
the 1-form $\omega$ has exactly one zero, and 
there are exactly $6g-5$ homotopy classes of systoles. 
\label{fig:hyperelliptic-saturate} 
 }\vspace{-20pt}
\end{center}
\end{figure}

\begin{proof}
Each homotopy class of systole is represented by at least one systolic
Delaunay edge. Since $\omega$ has exactly one zero, $z_0$,
the number of Delaunay edges is at most $6g-3$. Thus, we wish to show
that if there are $6g-3$ or $6g-4$ systolic Delaunay edges,
then there exist at least two homotopic pairs of systolic edges, and that, if there are
$6g-5$ systolic edges, then there is at least one pair of homotopic edges. 

\underline{$6g-3$ systolic edges:} \hspace{.2cm}
Suppose that there are exactly $6g-3$ systolic Delaunay edges.
Then each Delaunay 2-cell is an equilateral triangle and by Proposition \ref{prop:Delaunay-count}
there are $4g-2$ such cells. 
Since $\tau$ is an isometry, it preserves the Delaunay partition.
In particular, since $z_0$ is the unique $0$-cell, we have $\tau(z_0)=z_0$,
and since an equilateral triangle has no (orientation preserving) involutive isometry, the involution $\tau$ has no fixed points on the interior of each $2$-cell. 
Thus, the remaining $2g+1$ fixed points of $\tau$ lie on $1$-cells.
In particular, $\tau$ fixes exactly $2g+1$ Delaunay edges.

Suppose that $T$ is a 2-cell with two fixed edges. Then $T \cup \tau(T)$ 
is a cylinder whose boundary components are the `third' edges of $T$ and $\tau(T)$,
and, in particular, since the genus of $X$ is at least two, 
these `third' edges are not fixed by $\tau$.
Thus, a 2-cell has either zero, one, or two fixed edges.
Note that the number of 2-cells that have two fixed edges is even. 

We claim that there exist at least four 2-cells that each have two fixed edges. 
Indeed, if, on the contrary, there are at most two such 2-cells, 
then there are at least $4g-4$ remaining $2$-cells that each have at most one 
fixed edge. Thus, there are at most $2g-2$ fixed Delaunay edges associated to these 2-cells,
and at most $2$ edges associated to the $2$-cells that have two fixed edges. 
But, there are $2g+1 > (2g-2)+2$ fixed edges, and so we have a contradiction. 

The four 2-cells form two cylinders each bounded by two systolic edges. 
Thus, there are at most $6g-5$ homotopy classes of systoles. 

If there are exactly $6g-5$ homotopy classes of cylinders, then 
there are two maximal cylinders each bounded by two systolic edges. 
The integral of $\omega$ over the middle curve of each cylinder is nonzero,
and hence the middle curve is not null-homologous. The induced action 
of a hyperelliptic involution on $H_1(X)$ is the antipodal map, and so 
$\tau$ preserves each cylinder and has exactly two fixed points
on the interior of each cylinder. It follows that there are exactly four
Delaunay 2-cells each having two fixed edges.

\underline{$6g-4$ systolic edges:} \hspace{.2cm}
Suppose that there are exactly $6g-4$ systolic Delaunay edges.
It follows that exactly $4g-4$ Delaunay 2-cells are equilateral triangles.
The complement, $K$, of the union of these equilateral triangles is
(the interior of) a rhombus. 

Since $\tau$ is an isometry, $\tau$ preserves $K$, and hence 
the center $c$ of the rhombus is a fixed point of $\tau$. 
The other Delaunay 2-cells are equilateral triangles and hence
do not contain fixed points. Therefore, since $\tau$ has 
exactly $2g-2$ fixed points and $\tau(z_0)=z_0$, 
exactly $2g$ systolic edges are fixed by $\tau$.

If an edge $e$ in $\partial K$ is fixed by $\tau$, then $e$ is equal to the opposing edge
and in particular  $K \cup e$ is a cylinder. 
Indeed, if $e$ were fixed by $\tau$, then the
segment in $K$ joining the midpoint of $e$ to $c$ would 
be `rotated' by $\tau$ to a segment joining $c$ to the midpoint of the
edge $e'$ opposite to $e$. Hence the midpoint of $e$ would equal the midpoint of 
$e'$, and thus $e=e'$. 

Since $X$ is connected and of genus at least two, 
not all four edges of $\partial K$ can be fixed by $\tau$.
Thus either $\oK$ is a rhombus with no fixed edges or a cylinder with no fixed
boundary edges. 

Suppose that $\oK$ is a rhombus. Among the remaining $4g-4$ two-cells---equilateral 
triangles---there are exactly $2g$ fixed points. Hence there exist equilateral
triangles that have at least two fixed edges. If there were $4g-6$ equilateral triangles 
that each had at most one fixed edge, then there would be only 
$2g-3+2=2g-1$ fixed points among the $4g-4$ equilateral triangles. 
It follows that there are at least four equilateral triangles that each have 
two fixed edges, and thus there exist two distinct maximal cylinders bounded 
by systoles. 

Suppose that $\oK$ is a cylinder. In this case, neither of 
the two equilateral triangles that share edges with $\oK$ 
can have two fixed edges. Indeed, using an argument as above with a 
segment joining the center of the rhombus, we would see that the edges
would be identified in such a way to form a torus. 

Consider the two equilateral triangles $T_+$, $T_-$ that have an edge in $\oK$.
If an edge $e$ of $T_{\pm}$ is fixed then $\tau$, then using the symmetry about $c$, 
one shows as above that $e$ is identified with an edge of $T_{\mp}$. 
Because the $X$ is connected and of genus at least two, $T_{\mp}$ has 
at most one edge fixed by $\tau$. It follows that among
the remaining $4g-6$ triangles there are at least $2g+2 - 1 -3=2g-2$ fixed points.
It follows that there exist at least one equilateral that has two 
fixed edges, and hence there exists another maximal cylinder bounded by systoles. 

In either case, we have two maximal cylinders bounded by systoles, 
and therefore there are at most $6g-6$ 
homotopy classes of systoles.

\underline{$6g-5$ systolic edges:} \hspace{.2cm}
Suppose that there are exactly $6g-5$ systolic edges. 
Then there are $4g-6$ Delaunay 2-cells that are equilateral triangles. 
The complement, $K$, of the union of these equilateral triangles 
consists of either an equilateral hexagon or two disjoint rhombi.

Suppose that $K$ is an equilateral hexagon.
Then since $\tau$ preserves the Delaunay partition, 
we have $\tau(K)=K$. Hence $K$ contains
exactly one fixed point $c$ and $K$ is convex. Thus, arguing as above,
we find that if a boundary edge of $K$ is fixed by $\tau$, then 
the edge equals an opposite edge. Since $X$ 
is connected with genus at least two, all six edges
can not be identified, and hence there are at most $3$ fixed points
in $\overline{K}$. 

We claim that at least one pair of equilateral triangles each have 
exactly two fixed edges. If not, then each of the $4g-6$ equilateral
triangles contains at most one fixed edge. Thus, there are at most 
$2g-3$ such edges, and hence $(2g-3)+3+1=2g+1$ fixed points in total. 
But the total number of fixed points is $2g+2$. 
Thus, we have a pair of equilateral triangles that share a pair of fixed edges.
The union is a cylinder bounded by two systolic edges,
and so we have at most $6g-6$ homotopy classes of systoles in this case.

Finally suppose that $K$ is the disjoint union of two rhombi $R_+$ and $R_-$.
Since $\tau$ preserves the Delaunay partition, either $\tau(R_{\pm})=R_{\pm}$
or $\tau(R_{\pm})=R_{\mp}$. 

If $\tau(R_{\pm})=R_{\pm}$, then each rhombus contains a fixed point. 
If an edge of $R_{\pm}$ is fixed, then 
$\overline{R_{\pm}}$ is a cylinder bounded by systolic edges and so there
are at most $6g-6$ homotopy classes of systoles. 
If neither rhombus has boundary edges fixed by $\tau$, 
then $\overline{K}$ contains exactly two fixed points.
If there is not a pair of equilateral triangles that share fixed boundary edges,
then each of the $4g-6$ equilateral triangles would have at most one fixed edge,
and so there would be at most $2g-3+2+1=2g$ fixed points, a contradiction.
Hence we have a systolic cylinder and at most $6g-6$ homotopy classes 
of systoles. 

If $\tau(R_{\pm})=R_{\mp}$, then the rhombi do not contain fixed points.
If an edge in $\partial R_{\pm}$ is fixed by $\tau$, then $R_{\pm}$
shares this edge with $R_{\mp}$. It follows that there are at most three 
fixed points in $\overline{K}$, and one may argue as in the case of the hexagon, 
to find that there are at most $6g-6$ homotopy classes of systoles. 

If no edge in $\partial R_{\pm}$ is fixed by $\tau$, then among 
the $4g-6$ equilateral triangles there are $2g+1$ fixed points. 
It follows that there is an equlateral triangle that has two edges
fixed by $\tau$, and hence there is a maximal cylinder bounded by systoles.  
\end{proof}

Since each genus two surface is hyperelliptic, we have the following
corollary.

\begin{coro}
Let $X$ be a surface of genus two. If $\omega$ is a holomorphic 1-form on $X$ that has exactly one zero, then the number of homotopy classes of systoles of $(X,d_{\omega})$
is at most 7.
\end{coro}

%%%%%%%%%%%%%%%%%%%%%%%%%%%%%%%%%%%%%%%%%%%%%%%%%%%%%%%%%%%%%

\subsection{Lengths of systoles}

Although our main concern is the number of systoles, we observe in this section that it is quite straightforward to find a sharp upper bound on the length of systoles of translation surfaces provided they have a single cone point singularity. One of the ingredients 
is the Delaunay triangulation described in \S \ref{subsec:delaunay}. The other ingredient is a result due to Fejes T\'oth which we state in the form of the following lemma.

\begin{lem}\label{lem:triangle}
Let $T$ be a Euclidean triangle embedded in the plane and let $r$ be the 
maximal positive real number so that the open balls of radius $r$ around 
the three vertices are disjoint. Then
$$
r^2~ \leq~ \frac{{\rm Area}(T)}{\sqrt{3}}
$$
with equality if and only if $T$ is equilateral.
\end{lem}

This can be stated differently in terms of ratios of areas. 
Consider the area $A_r$ of a triangle found at distance $r$ 
from the vertices of $T$ and so that the interior of the three sectors do not overlap. 
Then the ratio $A_r/T$ never exceeds that of the equilateral triangle 
with $r$ equal to half the length of a side.

\begin{thm}
If $(X,\omega) \in \Hcal(2g-2)$, then
\begin{equation} \label{eqn:ratio-estimate}
\frac{\sys^2(X)}{\area(X)} \leq \frac{4}{(4g-2)\cdot \sqrt{3}}
\end{equation}
with equality if and only if $X$ is tiled by equilateral triangles.
\end{thm}

\begin{proof}
Let $z_0$ denote the zero of $\omega$. 
By Proposition \ref{prop:shortDelaunayedge}, each systole
that passes through $z_0$ lies in the 1-skeleton of the Delaunay cell decomposition 
of $(X, d_{\omega})$. Thus, if $r_0$ is the radius of the largest open Euclidean ball 
that can be embedded in $(X, d_{\omega})$ with center $z_0$,  then $r_0 = \sys(X)/2$.
Therefore
\begin{equation} \label{eqn:sys-one-zero}
  \sys^2(X)~ 
=~ 
4 \cdot r_0^2.
\end{equation}
Each open 2-cell $P$ of the Delaunay cell-deomposition is isometric to a convex Euclidean 
polygon. We may further subdivide each $2$-cell into Euclidean triangles to obtain
a triangulation of $X$ with one vertex, namely $z_0$. 
By Proposition \ref{prop:Delaunay-count}, there are at exactly $4g-2$ 
such triangles. Thus, Lemma \ref{lem:triangle} implies that 
\[
\area(X)~ =~ \sum_{T} \area(T)~
\geq~
(4g-2) \cdot \sqrt{3} \cdot r_0^2
\]
where the sum is over triangles in the triangulation.
By combining this estimate with (\ref{eqn:sys-one-zero}), we obtain the desired inequality 
(\ref{eqn:ratio-estimate}).
Moreover, if equality holds in (\ref{eqn:ratio-estimate}), 
then equality holds in Lemma  \ref{lem:triangle}, and 
so each triangle is an equilateral triangle. 
\end{proof}

We note that there is a unique surface (up to homothety) in $\Hcal(2)$ tiled by equilateral 
triangles (illustrated previously in Figure \ref{fig:genus2equilateral}). 
This surface realizes the maximum number of systoles and the maximum sytolic ratio over  $\Hcal(2)$.
In contrast, as indicated in the introduction, the maximum systolic ratio over $\Hcal(1,1)$ 
is not realized by the unique surface that realizes the maximum number of homotopy 
classes of systoles. 

%%%%%%%%%%%%%%%%%%%%%%%%%%%%

\section{Geodesics on a surface in $\Hcal(1,1)$}

In this section, $X$ will denote a $\Hcal(1,1)$ surface of genus two equipped with a translation structure with two cone points $c_+$ and $c_-$ each of angle $4 \pi$.
The tangent bundle of a translation surface is parallelizable. In particular, each oriented segment has a direction. The hyperelliptic involution $\tau: X \to X$ is an isometry that reverses the direction of each oriented segment. The isometry $\tau$ has exactly six fixed points, the {\em Weierstrass points}. 

\begin{lem} \label{cplusminus}
The hyperellipic involution
$\tau$ interchanges cone points: $\tau(c_{\pm}) = c_{\mp}$
\end{lem}

\begin{proof}
Since $\tau$ is an isometry the set $\{c_+, c_-\}$ is permuted. If $\tau(c_+)=c_+$, then in a neighborhood of $c_+$, the isometry $\tau$ acts as a rotation of $\pi$ radians. But the cone angle is $4 \pi$, and hence it is impossible for $\tau^2$ to be the identity.
\end{proof}

By Lemma \ref{cplusminus}, the quotient $X/\langle \tau\rangle$ is a sphere with one cone point $c^*$ with angle $4\pi$ and six cone points $\{c_1, \ldots, c_6\}$ each of angle $\pi$. Let $p: X \rightarrow X / \langle \tau \rangle$ denote the degree 2 covering map branched at $\{c_1, \ldots, c_6\}$. If $\gamma$ is a simple geodesic loop, then either $\gamma$
passes through two Weierstrass points in which case $p$ maps $\gamma$ onto a geodesic arc joining two distinct $\pi$ cone points, or $p \circ \gamma$ is a simple geodesic loop that misses the $\pi$ cone points.

A {\em flat torus} is a closed translation surface (necessarily of genus one). A {\em slit torus} 
is a flat torus with finitely many disjoint simple geodesic arcs removed.
 Each removed arc is called a {\em slit}. The completion of a slit torus (with respect to the 
natural length space structure) is obtained by adding exactly two geodesic segments for each 
removed arc. The interior angle between each pair of segments is $2 \pi$. 
This property characterizes slit tori.

\begin{lem} \label{lem:slit_torus}
Let $Y$ be a topological torus with a closed disc removed.
If $Y$ is equipped with a translation structure such that the
boundary\footnote{By {\em boundary} we mean the set of points
added by taking the metric completion of the length structure
associated to the translation structure.} component consists
of at most two geodesic segments, then $Y$ is isometric to a slit torus.
\end{lem}

\begin{figure}[h]
%\ShowGrid
%{\color{linkred}
\leavevmode \SetLabels
%\L(.39*.37) $\sqrt{3}$\\%
\L(.318*.65) $a$\\%
\L(.67*.81) $a$\\%
\L(.318*.32) $b$\\%
\L(.67*.16) $b$\\%
\L(.48*.05) $c$\\%
\L(.48*.92) $c$\\%
\L(.423*.48) $\theta_{+}$\\%
\L(.366*.48) $\theta_{-}$\\%
\L(.615*.69) $\theta_{-}$\\%
\L(.615*.29) $\theta_{-}$\\%
%\L(.398*.27) $r_0$\\
%\L(.595*.97) $\delta$\\
%\L(.49*-0.08) $\delta_{-}$\\
\endSetLabels
\begin{center}
\AffixLabels{\centerline{\epsfig{file =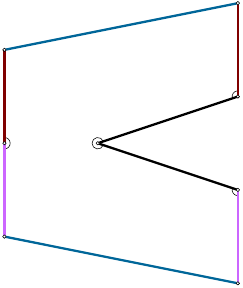,width=5cm,angle=0}}}
\vspace{-24pt}
\end{center}
\caption{Identify
the edges with the same labels via elements of ${\rm Isom}(\Rbb^2)$
to obtain a torus with a disc removed equipped with a flat structure
such that the boundary consists of exactly two geodesics. The angles
between the geodesics are not both $\pi$ though they sum to $4 \pi$.}\label{flat-non-slit-torus}
\end{figure}
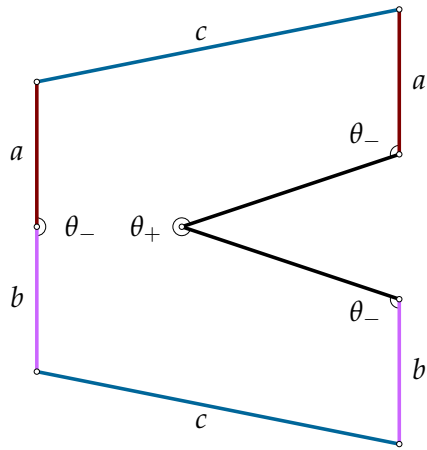

\begin{remk}
Figure \ref{flat-non-slit-torus} shows that
Lemma \ref{lem:slit_torus} is false if one replaces the assumption
of translation structure with the assumption of flat structure.
\end{remk}

\begin{proof}[Proof of Lemma \ref{lem:slit_torus}]
Let $Z$ be the boundary of $Y$.
Let $A$ be the intersection of the maximal geodesic segments in $Z$.
By assumption $A$ is either empty, contains one point, or contains two points.
Let $\alpha: [0,1] \to Z$ be a parameterization of $Z$ such that if $A$
is nonempty, then $\alpha(0)=\alpha(1) \in A$. Let $\oalpha$ be the development of $\alpha$ into the plane $\Cbb$ as discussed in \S \ref{section:lfacts}.

Since $[\alpha] \in \pi_1(Y)$ is a commutator and $\Cbb$ is abelian, the holonomy of $[\alpha]$ equals 0. Hence by (\ref{dev-equivariance}), we have $\oalpha(1)- \oalpha(0)=\dev([\alpha] \cdot \talpha(0))=0$, and therefore $\oalpha(1)=\oalpha(0)$.

If $A$ is empty or consists of one point, then $\oalpha$ is a line segment,
but this is impossible as line segments in $\Cbb$ have distinct endpoints.
If $A$ consists of two points, then the curve $\oalpha$ consists of two line segments.
Since $\oalpha(1)=\oalpha(0)$, the line segments coincide.
Removing this segment and its translates by $\hol(\pi_1(Y))$ and quotienting it by $\hol(\pi_1(Y))$ gives a surface isometric to $Y$.
\end{proof}

As a corollary, we have the following sharpening of Theorem 1.7 in
\cite{McMullen}.

\begin{coro} \label{coro:slit}
If $\alpha$ is a separating simple closed geodesic on $X$,
then $X - \alpha$ is the disjoint union of two slit tori.
Moreover, each slit torus contains exactly three Weierstrass points,
and the hyperelliptic involution $\tau$ preserves $\alpha$.
\end{coro}

\begin{proof}
Since $\alpha$ is separating and $X$ is closed of genus two,
the complement of $\alpha$ consists of two one-holed tori $Y_+$ and $Y_-$.
Since $\alpha$ is geodesic, the boundaries of $Y_+$ and $Y_-$ are
piecewise geodesic. Since $\alpha$ is simple and there are only two
cone points, the number of geodesic pieces of $Y_{\pm}$ is at most two.
Lemma \ref{lem:slit_torus} implies that each component is a slit torus.

The restriction of $\tau$ to a slit torus component
determines an elliptic involution $\underline{\tau}$ of the torus.
The endpoints of each slit correspond to the cone points $c_+$ and $c_-$,
and so the are preserved by the induced elliptic involution.
Since $\tau$ preserves the cone points, the map $\underline{\tau}$
preserves the slit, and hence $\alpha$ is preserved by $\tau$.
In particular, the midpoint of the slit is fixed
by $\underline{\tau}$ and the three other fixed points of $\underline{\tau}$
are fixed points of $\tau$.
\end{proof}

A {\it cylinder} of girth $\ell$ and width $w$ is an isometrically embedded
copy of $(\Rbb/\ell\Zbb) \times [-w/2,w/2]$.
Each cylinder is foliated by geodesics indexed by $t \in [-w/2,w/2]$.
We will refer to the geodesic that corresponds to $t=0$ as the
{\em middle geodesic}. By Corollary \ref{coro:slit}, if a simple closed
geodesic lies in a cylinder, then it is nonseparating.

A cylinder $C$ is said to be {\it maximal}
if it is not properly contained in another cylinder.
If a closed translation surface has a cone point, then each geodesic
that does not pass through a cone point lies in a unique maximal cylinder.

Because the hyperelliptic involution $\tau$ reverses the orientation
of isotopy classes of simple curves, the map $\tau$ restricts to
an orientation reversing isometry of each maximal cylinder $C$, and thus it
restricts to an orientation reversing isometry of the middle geodesic
$\gamma \subset C$. In particular, it contains two Weierstrass points.

\begin{prop} \label{prop:nonseparating}
If $\gamma$ is a nonseparating simple closed geodesic, then $\gamma$ is homotopic
to a unique geodesic $\gamma'$ such that the restriction of $\tau$
to $\gamma'$ is an isometric involution of $\gamma'$.
\end{prop}

\begin{proof}
If $\gamma$ does not contain a cone point, then $\gamma$ belongs to
a maximal cylinder. If $\gamma$ belongs to a maximal cylinder $C$, then
it is homotopic to the middle geodesic $\gamma' \subset C$.

If $\gamma$ does not belong to a cylinder, then $\gamma$ is the unique
geodesic in its homotopy class. Since $\tau$ reverses the orientation of
the homotopy classes of simple loops, it acts like an orientation reversing
isometry on $\gamma$.
\end{proof}

Proposition \ref{prop:nonseparating} reduces the counting of homotopy
classes of nonseparating systoles to a count of nonseparating systoles
that pass through exactly two Weierstrass points.
In the next two sections we analyse such geodesics.

%see
%%%%%%%%%%%%%%%%%%%%%%%%%%%%%%%%%%%%%%%%%%%%%%%%%%%%%%%%%%%%%%

\section{Direct Weierstrass arcs} \label{sec:direct-arcs}
 
If $\gamma$ is a simple closed
geodesic on $X$ that passes through two Weierstrass points,
then the projection $p(\gamma)$ is an arc on $X / \langle \tau\rangle$
that joins one angle $\pi$ cone point to another angle $\pi$ cone point.
We will call each such an arc a {\em Weierstrass arc}. Note that the $p$ inverse image
of a Weierstrass arc is a geodesic and so we obtain a one-to-one correspondence
between homotopy classes of nonseparating simple geodesic loops on $X$ and
Weierstrass arcs on $X/ \langle \tau\rangle$. A Weierstrass arc that is the image of
a systole will be called a {\em systolic Weierstrass arc}. Note that 
each systolic Weierstrass arc has length equal to $\sys(X)/2$.
 
The Weierstrass arcs come in two flavors. We will say that a Weierstrass arc is
{\em indirect} if it passes through the angle $4\pi$ cone point, and otherwise
we will call it {\em direct}.

\begin{lem} \label{lem:at_most_one_direst_arc}
There is at most one direct systolic Weierstrass arc joining two angle $\pi$ cone points.
\end{lem}

\begin{proof}
Suppose to the contrary that there exist two distinct direct systolic Weirestrass arcs
that both join the angle $\pi$ cone point $c$ to the angle $\pi$ cone point $c' \neq c$. 
These arcs lift to closed systoles $\gamma_+$ and $\gamma_-$ 
that interesect transversally at two Weierstrass points corresponding to $c$ and $c'$. 
In particular, the Weierstrass points divide each geodesic into two arcs. 
By concatenating a shorter\footnote{If the arcs have the same length, then choose either arc.} 
arc of $\gamma_+$ with a shorter arc of $\gamma_-$ we construct a piecewise geodesic
closed curve $\alpha$ that has length at most the systole.
Since the angle between the arcs is strictly between $0$ and $\pi$, 
we can perturb $\alpha$ to obtain a shorter curve whose length
is strictly less than the systole. This contradicts the assumption
that $\gamma_+$ and $\gamma_-$ are both systoles.
\end{proof}

\begin{remk}\label{remk:surgery}
The argument in the above lemma was that the concatenation of two geodesic arcs that meet with an angle strictly less than $\pi$ cannot be of minimal length in their homotopy class. In particular, they can't form a systole. This argument will be used several times. 
\end{remk}

\begin{prop} \label{prop:intersection-direct-W-arcs}
Let $\gamma_+$ and $\gamma_-$ be distinct nonseparating systoles on $X$.
If each contains two Weierstrass points and neither contains 
a $4\pi$ cone point, then the intersection $\gamma_+ \cap \gamma_-$
is either empty or consists of a single Weierstrass point. In particular, 
the geometric intersection number $i(\gamma_+, \gamma_-)$ equals either zero or one. 
\end{prop}

\begin{proof}
Each projection $\alpha_{\pm} =p(\gamma_{\pm})$ is a direct systolic Weierstrass arc. 
By Lemma \ref{lem:at_most_one_direst_arc}, at most one angle $\pi$ cone point 
lies in the intersection $\alpha_+ \cap \alpha_-$, and hence $\gamma_+ \cap \gamma_-$ 
contains at most one Weierstrass point. 

Suppose (to the contrary) that the intersection $\gamma_+ \cap \gamma_-$ were to contain
a point on $X$ that is not a Weierstrass point. Then $\alpha_+ \cap \alpha_-$
would contain a point $p$ that is not an angle $\pi$ cone point. Since, $\gamma_{\pm}$
is a systole, there would exist a subarc, $\beta_{\pm}$,of $\alpha_{\pm}$ that joins 
$p$ to an endpoint of $\alpha_{\pm}$ whose length is at most $\sys(X)/4$. 
By concatenating $\beta_+$ and $\beta_-$ and perturbing, we would obtain an arc 
joining two angle $\pi$ cone points whose length would be strictly less than $\sys(X)/2$. 
This arc would lift to a closed curve on $X$ whose length is less than $\sys(X)/2$,
a contradiction.  
\end{proof}

The following result is central to the proof of Theorem \ref{thm:main}.

\begin{thm} \label{prop:bivalent-direct}
If $c$ is a cone point on $X/\langle \tau\rangle$ with angle $\pi$, then at most
two direct systolic Weierstrass arcs have an endpoint at $c$.
Thus, there are at most six direct systolic Weierstrass arcs.
\end{thm}

The remainder of this section is devoted to the proof of Theorem \ref{prop:bivalent-direct}.
The proof is a complicated proof by contradiction that involves many cases. 
We suppose that there exist three direct systolic Weierstrass arcs 
that end at $c$. We cut along these arcs and we cut along the two 
(necessarily direct)
minimal arcs that join the remaining two angle $\pi$ cone points to the angle 
$4 \pi$ cone point on $X/ \langle \tau \rangle$. The result of these cuts is an annulus 
with piecewise geodesic boundary that contains the remnants of the cone points. 
The various cases considered are based on the holonomy of the translation
structure of the annulus as well as the relative positions of the cone points
on the boundary of the annulus. To obtain a contradiction in each case, we use
the fact that the distance between any two cone points can be no less than
$\sys(X)/2$. 

We now begin the proof of  Theorem \ref{prop:bivalent-direct}.

\begin{proof}
Suppose to the contrary that there exist three direct systolic Weierstrass 
arcs each having $c$ as an endpoint. Let $\theta_1 \leq \theta_2 \leq \theta_3$
denote the angles between the arcs at $c$. Since $c$ is an angle $\pi$ cone point, we have
$\theta_1 + \theta_2 + \theta_3 = \pi$. Label the arcs $\alpha_i$, $i \in \Zbb/3\Zbb$,
so that the angle between $\alpha_{i-1}$ and $\alpha_{i}$ equals $\theta_i$.
By Lemma \ref{lem:at_most_one_direst_arc}, the other endpoints of the $\alpha_i$
are all distinct. Label the other endpoint of $\alpha_i$ with $c_i$.
Let $c_4$ and $c_5$ denote the two remaining angle $\pi$ cone points.

The lift, $\talpha_i$, of each $\alpha_i$ to $X$ is a non-separating direct simple 
closed geodesic on $X$. The involution preserves $G:=\talpha_1 \cup \talpha_2 \cup \talpha_3$ 
and hence the complement $A:=X-G$.
We have $\chi(A) = \chi(X)- \chi(G)= 2-2=0$, and since $A$
contains the fixed points $c_4$ and $c_5$, it follows that
$A$ is connected and, moreover, is homeomorphic to an annulus. 

Let $\gamma$ be a shortest geodesic in $X$ that represents the free homotopy
class corresponding to a generator of $\pi_1(A) \subset \pi_1(X)$.
Because $\theta_i< \pi$ and each $\talpha_i$ is a geodesic, 
the geometric intersection number of $\gamma$ and each $\talpha_i$ is zero. 
In particular, $\gamma$ can not coincide with some
$\talpha_i$ as the intersection number $i(\talpha_i, \talpha_j) =1$ for $i \neq j$
(see Proposition \ref{prop:intersection-direct-W-arcs}). 
Therefore, $\talpha_i$ and $\gamma$ are disjoint for each $i \in \Zbb/3\Zbb$,
and $\gamma$ lies in $A$. 

In the remainder of the proof, we will consider separately 
the two cases: (1) the closed geodesic $\gamma$ is direct and 
(2) $\gamma$ passes through
an angle $4 \pi$ cone point. 

\underline{$\gamma$ is direct:} \vspace{.2cm} If $\gamma$ is direct, 
then it belongs to a maximal cylinder $C$. Without loss of generality, 
$\gamma$ is the middle geodesic of this cylinder. 
Since $\gamma$ is nonseparating, $\tau$ preserves $C$
and $\gamma$, and in particular, the fixed points $c_4$ and $c_5$
lie on $\gamma$. To obtain the desired contradiction in this case, 
it suffices to show that the length of $\gamma$ is less than $\sys(X)$.

Each component of $\partial C$ consists of a direct geodesic segment
$\beta_{\pm}$ joining an angle $4 \pi$ cone point $c_{\pm}^*$ to itself.
The geometric intersection number of $\beta_{\pm}$ and each
$\talpha_i$ equals zero, and hence $\beta_{\pm}$ does not
intersect any of the $\talpha_i$. 
Hence the complement $A - C$ consists of two topological 
annuli $K_+$ and $K_-$ with 
$\beta_{\pm} \subset \partial K_{\pm}$.
Because $\tau$ preserves each maximal cylinder as well as $A$,
we have $\tau(K_{\pm})= \tau(K_{\mp})$. Thus, we will now limit our attention to only one of the two annuli,
$K:=K_+$. One boundary component of $K$ is the direct 
geodesic segment $\beta:=\beta_+$ joining an angle 
$4 \pi$ cone point, $c^*:=c^*_+$, to itself.
The other boundary component, $\beta'$, of $K$ consists of
three geodesic segments $\oalpha_1$, $\oalpha_2$, and $\oalpha_3$ corresponding respectively
to $\talpha_1$, $\talpha_2$, and $\talpha_3$.
Moreover, the interior angle between $\oalpha_{i-1}$ and $\oalpha_{i}$ is
equal to $\theta_i$. 
See the left hand side of Figure \ref{Fig-A-annulus}.

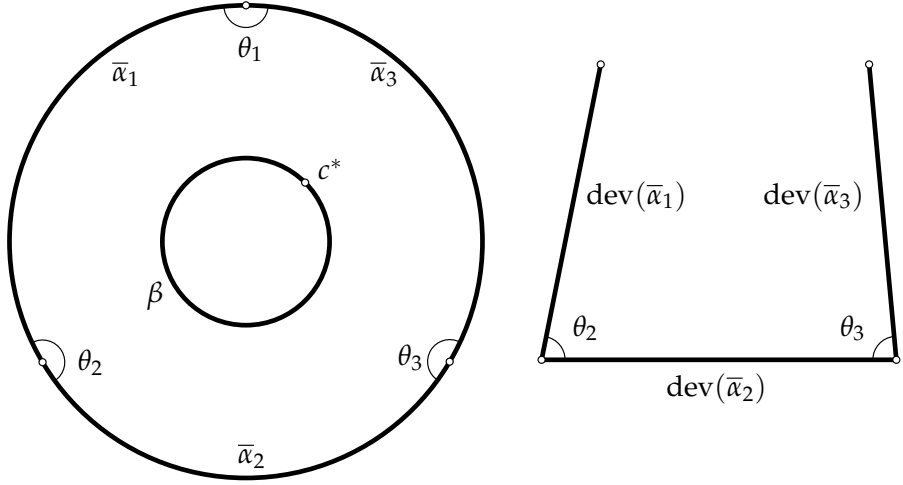
\begin{figure}[h]
%\ShowGrid
%{\color{linkred}
\leavevmode \SetLabels
%\L(.39*.37) $\sqrt{3}$\\%
\L(.38*.63) $c^{*}$\\%
\L(.23*.38) $\beta$\\%
\L(.31*.88) $\theta_{1}$\\%
\L(.17*.24) $\theta_{2}$\\%
\L(.45*.25) $\theta_{3}$\\%
\L(.20*.83) $\oalpha_1$\\%
\L(.31*.05) $\oalpha_2$\\%
\L(.426*.83) $\oalpha_3$\\%
\L(.603*.31) $\theta_{2}$\\%
\L(.837*.31) $\theta_{3}$\\%
\L(.615*.58) $\dev(\oalpha_1)$\\%
\L(.685*.187) $\dev(\oalpha_2)$\\%
\L(.77*.58) $\dev(\oalpha_3)$\\%
%\L(.398*.27) $r_0$\\
%\L(.595*.97) $\delta$\\
%\L(.49*-0.08) $\delta_{-}$\\
\endSetLabels
\begin{center}
\AffixLabels{\centerline{\epsfig{file =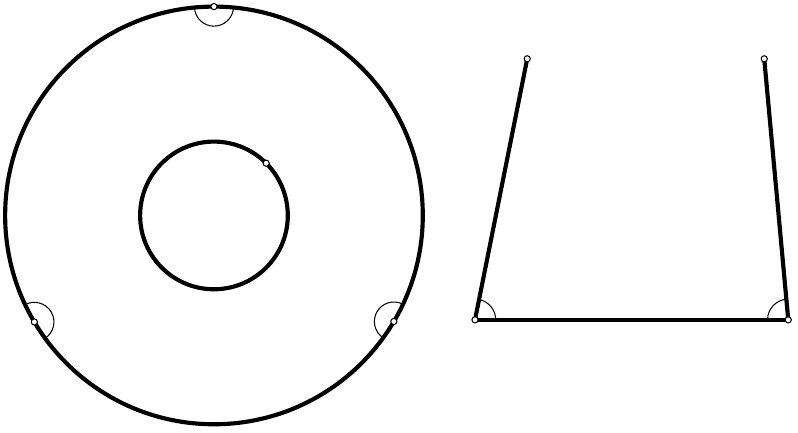,width=12cm,angle=0}}}
\vspace{-24pt}
\end{center}
\caption{On the left is the topological annulus $K$ case when the closed geodesic $\gamma$ is direct. The right side shows the development of 
$\beta'=\oalpha_1\cup \oalpha_2\cup \oalpha_3$. }\label{Fig-A-annulus}\end{figure}

Since $\beta$ and $\gamma$ are parallel geodesics in the same 
cylinder $C$, it suffices to show that the length of $\beta$ 
is less than $\sys(X)$.
Since $\beta$ is a direct geodesic segment, 
the length of $\beta$ equals the length of the holonomy 
vector associated to $\beta$. Since $\beta$ and $\beta'$
are homotopic, their holonomy vectors have the same length.
Thus, it suffices to show that the length of the holonomy
vector associated to $\beta'$ is less than $\sys(X)$.

Since, by assumption, each $\talpha_i$ is a systole, 
the length of $\beta'$ is $b:=3 \cdot \sys(X)$. 
Let $\beta':[0,b] \to \partial_{\pm}$ be a 
parameterization of $\beta'$ so that $\beta'(0)= \oalpha_3 \cap \oalpha_1= \beta'(1)$.
The development, $\obeta'$, consists of three 
line segments, each of length $\sys(X)$,
joined end to end with consecutive angles 
$\theta_2$ and $\theta_3$. See the right hand side 
of Figure \ref{Fig-A-annulus}.

Since $2 \pi/3 \leq \theta_2 + \theta_3< \pi$ and the three 
sides of $\obeta'$ have the same length, an elementary 
fact from Euclidean geometry applies to give that the distance between
$\dev(\beta'(0))$ and $\dev(\beta'(1))$ is less than $\sys(X)$.
Thus the holonomy vector of $\beta'$ has 
length less than $\sys(X)$ as desired. 

\underline{$\gamma$ is indirect:} \hspace{.2cm}In the remainder of the proof we consider the case in which $\pi_1(A)$ is not 
generated by a direct simple closed geodesic. In this case, 
the shortest geodesic $\gamma$ that generates $\pi_1(A)$ is 
unique in its homotopy class. In particular, since
$\tau$ induces a nontrivial 
automorphism of $\pi_1(A) \cong \Zbb$,
the isometry $\tau$ preserves $\gamma$ and reverses its orientation. It follows that $\gamma$ is a union of two geodesic segments each joining
the two $4\pi$ angle cone points, and each segment 
contains as its midpoint one of the remaining
two Weierstrass points. Let $\sigma_+$ denote the 
segment containing $c_4$, and let $\sigma_-$ 
denote the segment containing $c_5$.

The complement of $\gamma$ consists of two topological annuli $K_+$ and $K_-$
that are isometric via $\tau$. We limit our attention to one of the annuli, $K$.
One boundary component
of $K$ consists of the geodesic segments 
$\oalpha_1$, $\oalpha_2$, and $\oalpha_3$ with
the interior angle between $\oalpha_{i-1}$ and $\oalpha_{i}$ equal to $\theta_i$.
The other boundary component consists 
of $\sigma_+$ and $\sigma_-$. See Figure 
\ref{Fig-A-annulus-2}.
 
Let $c_+^*$ and $c_-^*$ denote the angle $4 \pi$ cone points. Let $\theta_{\pm}$ denote the
interior angle between $\sigma_+$ and $\sigma_-$ at $c_{\pm}^*$.
Because $\tau$ interchanges the two components of $A -\gamma$,
we have $\theta_+ + \theta_- = 4 \pi$. Since $\gamma$ is not direct, there is no
direct geodesic segment joining $c_4$ and $c_5$ inside $K$. Indeed, if there were such a
segment $\delta$, then $\delta \cup \tau(\delta)$ would be a direct simple closed geodesic
that generates $\pi_1(A)$ contradicting our assumption. It follows that $\theta_{\pm} \geq \pi$.

\begin{figure}[h]
%\ShowGrid
%{\color{linkred}
\leavevmode \SetLabels
%\L(.39*.37) $\sqrt{3}$\\%
\L(.456*.40) $c^{*}_{-}$\\%
\L(.523*.578) $c^{*}_{+}$\\%
\L(.337*.70) $c_1$\\
\L(.49*.04) $c_2$\\%
\L(.64*.72) $c_3$\\%
\L(.435*.66) $c_4$\\%
\L(.55*.335) $c_5$\\%
\L(.49*1.02) $v_1$\\
\L(.310*.21) $v_2$\\%
\L(.67*.21) $v_3$\\%
\L(.49*.88) $\theta_{1}$\\%
\L(.35*.29) $\theta_{2}$\\%
\L(.63*.3) $\theta_{3}$\\%
\L(.56*.65) $\theta_{+}$\\%
\L(.425*.31) $\theta_{-}$\\%
\L(.375*.933) $\oalpha_1$\\%
\L(.389*.12) $\oalpha_2$\\%
\L(.665*.5) $\oalpha_3$\\%
\endSetLabels
\begin{center}
\AffixLabels{\centerline{\epsfig{file =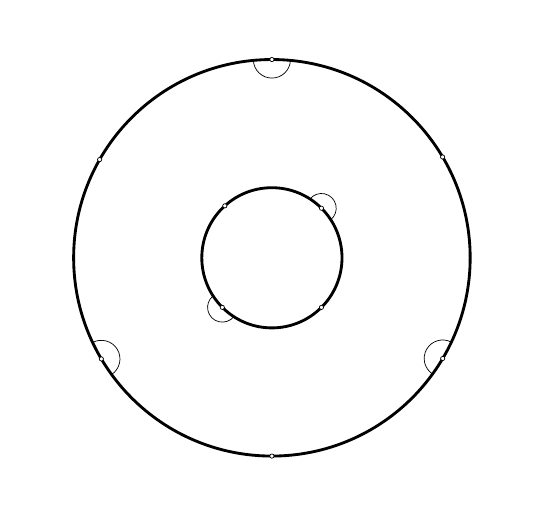,width=6cm,angle=0}}}
\vspace{-24pt}
\end{center}
\caption{The topological annulus $K$. \label{Fig-A-annulus-2}}
\end{figure}
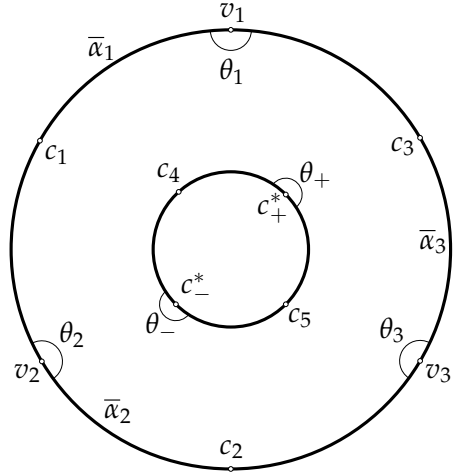

We claim that $\theta_1< \pi/3$.
Indeed if not, then since $\theta_1 + \theta_2+\theta_3= \pi$ 
and $\theta_1\leq \theta_2 \leq \theta_3$ , we would have
$\theta_i= \pi/3$ for each $i$ and in particular, 
the holonomy of $\beta=\talpha_1 \cup \talpha_2 \cup \talpha_3$ 
would be zero. Thus, since $\sigma_+ \cup \sigma_-$ 
is homotopic to $\beta$, the holonomy of $\sigma_+ \cup \sigma_-$
would be trivial. Since $\sigma_{\pm}$ is a geodesic segment, 
the angle at $c_{\pm}^*$ would equal $2\pi$ and the 
lengths of $\sigma_+$ and $\sigma_-$ would be equal. 
It would follow that the developing map would
map $\overline{K}$ onto the an equilateral triangle $T$
having sidelength $\sys(X)$. Moreover, 
$\dev(\sigma_+)= \dev(\sigma_-)$ would be a segment $\sigma$ in 
the interior of $T$ and the restriction of $\dev$ to 
$\overline{K} - (\sigma_+ \cup \sigma_-)$ would be injective. 
By elementary Euclidean geometry, the distance from each 
interior point of $T$ to the set of midpoints of the sides of 
$T$ is less than $\sys(X)/2$. In particular, it would
follow that there would be a direct 
geodesic segment in $\overline{K}$ joining the 
set $\{c_4, c_5\}$ and $\{c_1,c_2,c_3 \}$ having length 
less than $\sys(X)/2$. This would 
contradict the definition of $\sys(X)$.

Thus, in the remainder of the proof of Theorem \ref{prop:bivalent-direct}, 
we may assume that $\theta_1 < \pi/3$. Our next goal is the show that this
implies that there exists a direct geodesic that joins $v_1$
to one of the two $4\pi$ cone points, $c_{\pm}^*$. 

\begin{lem}  \label{lem:short-on-top}
There exists a (direct) geodesic segment $\delta \subset K$ that joins $v_1$ to either 
$c_+^*$ or $c_-^*$. 
\end{lem}

\begin{proof}[Proof of Lemma \ref{lem:short-on-top}]
Let $V$ be the set of points $x \in K$ such that
there exists a direct geodesic segment in $K$ joining $v_1$ to $x$.
By lifting to $\tX$ and applying the developing map, the set $V$
is mapped injectively onto a subset of the Euclidean sector $S$ 
of angle $\theta_1$. In particular, $v_1$ is mapped to the 
vertex $\ov_1$ of $S$. The bounding rays of $S$ 
contain the respective 
images, $\oc_1$ and $\oc_3$, of the points $c_1$ and $c_3$.

Let $T$ be the convex hull of $\{\ov_1, \oc_1, \oc_3\}$
The set $T$ is an isoceles 
triangle with $|\ov_1 \oc_1| = \sys(X)/2= |\ov_1 \oc_3|$,
and the angle $\angle \oc_1 \ov_1 \oc_3$ is less than $\pi/3$. 
In particular, the side of $T$ that joins $\oc_1$ and $\oc_3$ 
has length less than $\sys(X)/2$, and the distance from
$\ov_1$ to any other point of $T$ is at most $\sys(X)/2$.

Let $x^* \in \overline{S-V}$ be a point such that $\dist(x^*,v_1)$ 
equals the distance between $\ov_1$ and the $\overline{S-V}$. 
We claim that $x^*$ is the image of an angle $4\pi$ cone point,
and hence that there exists a direct geodesic joining $v_1$ and
this angle $4 \pi$ cone point. See Figure \ref{fig:triangle-T}.

\begin{figure}
\begin{center}
\includegraphics[width=6cm]{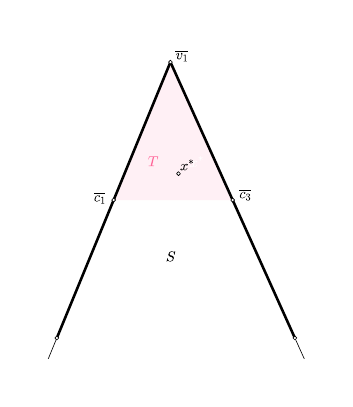}
\end{center}
\caption{The point $x^*$ in the triangle $T$. \label{fig:triangle-T}}
\end{figure}

To verify the claim, we first note that $x^*$ lies in the interior of $T$. 
Indeed if it did not, then since the developing map is injective
on $V$, the side of $T$ that joins $\oc_1$ to $\oc_3$ 
would be the image of a direct geodesic segment joining 
$c_1$ and $c_3$ having length less than $\sys(X)/2$.
This would contradict the definition of $\sys(X)$.

Because $\theta_1 \leq \theta_2 \leq \theta_3$, the distance 
between $v_1$ and $\talpha_2$ is at least $\sys(X)/2$, and 
hence the point $x^*$ can not belong to $\dev(\talpha_2)$.
Thus, $x^*$ is the image of a point in $\sigma_+$ or 
$\sigma_-$. Thus to verify the claim, 
it suffices to show that $x^*$ is not the image 
of an interior point of $\sigma_{\pm}$. 

Suppose to the contrary that $x^*$ were the image of an interior point
$\sigma_{\pm}$. Then the segment $\dev(\sigma_{\pm})$ would 
lie in $\overline{S-V}$, and hence by the definition of
$x^*$, the segment $\dev(\sigma_{\pm})$ would be perpendicular
to the segment joining $v_1$ and $x^*$, and hence 
parallel to the side of $T$ that opposes $v_1$. 
The segment $\dev(\sigma_{\pm})$ does not 
intersect either $\dev(\talpha_1)$ or $\dev(\talpha_3)$,
and hence the midpoint of $\dev(\sigma_{\pm})$ would lie 
in $T$. The segment joining the midpoint and $\ov_1$
corresponds to a direct geodesic segment joining $v_1$ 
to either $c_4$ or $c_5$. Since this segment has length less 
than $\sys(X)/2$, we would obtain a contradiction. 

Thus, $x^*$ is the image of either $c_-^*$ or $c_+^*$.
\end{proof}

By relabeling if necessary, we may assume that 
$\dev(c_+^*)= x^*$. Let $\delta$ denote the direct 
geodesic joining $v_1$ and $c_+^*$.

Let $P$ denote the metric completion of $K-\delta$. 
The metric space $P$ is a topological disk bounded by seven geodesic segments.
The `polygon' $P$ has seven vertices: the points $v_2$ and $v_3$, 
two vertices, $p_{+}$ and $p_-$, corresponding to $c_+^*$,
one vertex, $q$, corresponding to $c_-^*$,  and
two vertices, $v_+$ and $v_-$, corresponding to $v_1$. 
See Figure \ref{fig:polygon-P}.

\begin{figure}
\begin{center}\includegraphics[width=6cm]{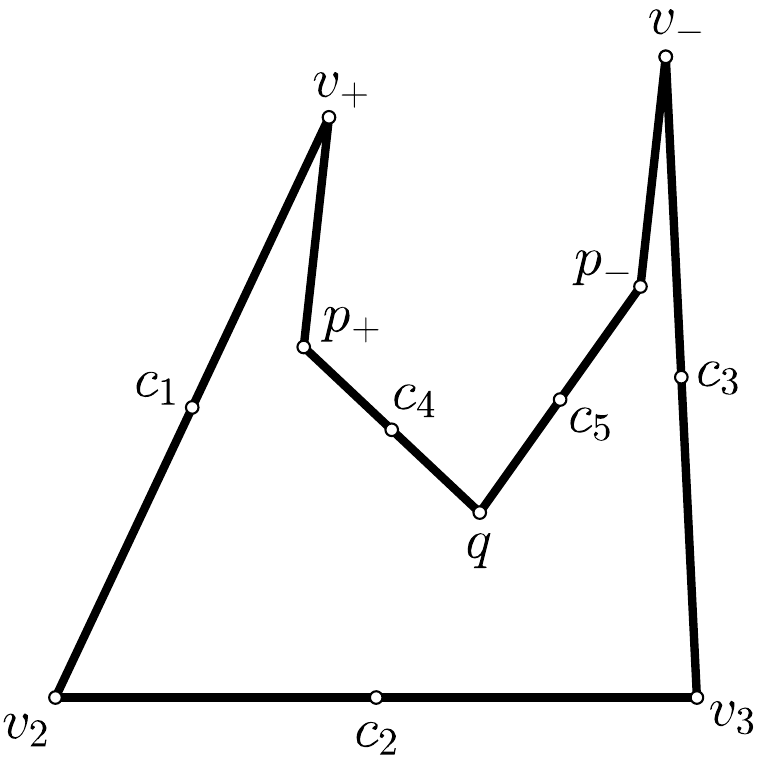}
\end{center}
\caption{The polygon $P$. \label{fig:polygon-P}}
\end{figure}

Continuing with our proof of Theorem \ref{prop:bivalent-direct}, our next goal
is to prove that $P$ may be regarded as a polygon in the plane.
In particular, we wish to show that the restriction of the developing map to $P$ is 
injective.\footnote{One may regard $P$ as a subset of $\tX$ by lifting
its interior and then taking the closure.}  

First, note that since the geodesics $\talpha_i$ all have the same length
and the sum of the angles $\theta_2+ \theta_3$ is strictly larger than $\pi/3$,
the set $\dev(\talpha_1 \cup \talpha_2 \cup \talpha_3)$ is a simple piecewise 
linear arc in the plane with endpoints $\ov_+$ and $\ov_-$ corresponding to 
$v_+$ and $v_-$ respectively. In particular, the convex hull of $\{\ov_+, \ov_2,\ov_3, \ov_-\}$
is a quadrilateral $Q$, and the line segments $\oalpha_i:= \dev(\talpha_i)$ 
constitute three of the sides of $Q$. 

Let $\delta_{\pm} \subset P$ be the segment that joins $v_{\pm}$ and $p_{\pm}$,
and let $\odelta_{\pm}: \dev(\delta_{\pm})$.  Since $\theta_1< \pi/3$,
the angle between $\talpha_1$ and $\delta_+$ and the angle between $\talpha_3$
and $\delta_-$ are both less than $\pi/3$. It follows that the segment 
$\odelta_{\pm}$ lies in $Q$ and that the point  $p_{\pm}$ lies in the interior of $Q$.

Let $\theta_z$ denote the interior angle at a vertex  $z$ of  $P$.

\begin{lem} \label{lem:interior-angles}
We have $\pi<\theta_q < 2 \pi$,  $\theta_{p_{\pm}} < 2 \pi$, and 
$\theta_{p_+} + \theta_{q}+\theta_{p_{-}} =  4\pi$.
\end{lem}

\begin{proof}
Suppose to the contrary that $\theta_q \leq  \pi$. 
Since the angles $\theta_{v_+}, \theta_{v_-}, \theta_2, \theta_3$ are all
less than $\pi$, the shortest path from $p_+$ to $p_-$ and the shortest path from $c_4$ to $c_5$
are both direct.  Because $\theta_2 + \theta_3 <\pi$, the Euclidean distance 
$|\op_+ \op_-|= |p_+p_-|$ is strictly less 
than $|\ov_2 \ov_3| = \sys(X)$. Since $c_4$ and $c_5$ are midpoints, it follows that 
$|c_4c_5|= |\oc_4\oc_5|< \sys(X)/2$. This is a contradiction. 

It follows that the point $\oq:=\dev(q)$ lies in the closed half-plane 
bounded by the line through $\op_+$ and $\op_-$ that does not contain $\ov_+$ or $\ov_-$. 
Hence, the angle $\theta_{p_{\pm}}$ is at most the angle between $\odelta_{\pm}$ 
and $\overline{p_+p_-}$, and this is less than $2 \pi$. The angle $\theta_q$
equals $2 \pi - \psi$ where $\psi$ the angle opposite the segment $\op_+\op_-$
in the (perhaps degenerate)  triangle $\op_+ \oq \op_-$.

As discussed in the analysis of Figure \ref{Fig-A-annulus-2}, 
we have $\theta_++\theta_- = 4\pi$. It follows that  
$\theta_{p_+} + \theta_q+ \theta_{p_-}= 4 \pi$.  
\end{proof}

\begin{prop}
The metric space $P$ is isometric to a simply connected polygon in the Euclidean plane.
\end{prop}

\begin{proof}
It suffices to show that the developing map is injective. 
Let $x, x' \in P$. Since $P$ is path connected and compact,
there exists a minimal geodesic arc $\eta$ that joins $x$ to $x'$. 
To prove the claim it suffices to show that the endpoints of $\dev \circ \eta$ 
are distinct. If $\eta$ is a direct geodesic segment, then 
$\dev \circ \eta$ is a single Euclidean line segment and so 
$\dev(x) \neq \dev(x')$. If $\eta$ is not direct, then $\eta$
is a concatenation of a finite number direct geodesic segments, 
$\gamma_1, \ldots, \gamma_n$, such that $\gamma_i \cap \gamma_{i+1}$
is a vertex $v_i$ and the angle $\psi_i$ 
between $\gamma_i$ and $\gamma_{i+1}$ satisfies 
$\pi \leq \psi_i \leq \theta_v$ where $\theta_v$ is the angle between 
boundary segments at $v$. Since the angles at $v_{\pm}$, $v_2$, and 
$v_3$ are less than $\pi$, the minimal geodesic $\eta$ can only pass through
the vertices $p_+$, $p_-$, or $q$, and if $\eta$ does pass through
one of these vertices, then it passes through the vertex at most once. 

Each of the angles $\theta_q$, $\theta_{p_{\pm}}$ is positive, and so
if $\eta$ passes through exactly one of the points $q, p_{\pm}$, then the path $\dev \circ \eta$ 
is a simple arc. In particular, the endpoints $\dev(x)$ and $\dev(x')$ are distinct.

Suppose that $\eta$ passes through exactly two vertices say $v_1,v_2 \in \{p_+,p_-,q\}$.
Lemma \ref{lem:interior-angles} implies that $\psi_1+ \psi_2 < 3 \pi$. We also have $\psi_i \geq \pi$. 
An elementary argument in Euclidean geometry shows that $\dev(\eta)$ is a simple arc.

Finally, suppose that $\eta$ passes through each of $p_+,p_-,q$.

Hence $\psi_1+ \psi_2 + \psi_3 \leq 4 \pi$. We also have $\psi_i \geq \pi$. An elementary 
Euclidean geometry argument shows that $\dev \circ \eta$ is a simple arc. 
\end{proof}

In what follows, we will identify the polygon $P$ 
with its image in $\Cbb$. See Figure \ref{fig:polygon-P}.

\begin{lem} 
The shortest geodesic joining 
$c_1$ (resp. $c_3$) to $p_+$ (resp. $p_-$) is direct.
\end{lem}

\begin{proof}
Recall the triangle $T$ described in Figure \ref{fig:triangle-T}. The point 
$p_+$ corresponds to $x^*=c_+^*$, and so if the shortest 
geodesic joining $c_1$ and $p_+$ 
were not direct, then the shortest geodesic in $X$
joining $c_1$ to $c_+^*$ would also pass through $c_-^*$. 
Hence $c_-^*$ would also belong to the triangle $T$ 
described above, and so either the image of $\sigma_+$ 
or the image of $\sigma_-$ would lie in $T$. But then the 
midpoint $c_4$ of $\sigma_+$ or the midpoint $c_5$ of 
$\sigma_-$ would belong to $T$. Hence $|v_1c_4|$ or $|v_1 c_5|$ 
would be less than $\sys(X)/2$, a contradiction. 

A similar argument shows that the shortest geodesic from 
$c_3$ to $p_-$ is direct.
\end{proof}

Because $x^*$ belongs to the interior of $T$, we have 
$\angle \ov_1 \oc_1 x^* < \angle \ov_1 \oc_1 \oc_3$.
Since $T$ is isoceles, we have 
$2 \cdot \angle \ov_1 \oc_1 \oc_3 + \theta_1 = \pi$.
Thus, it follows that 
\begin{equation} \label{est:cpv1}
\angle\, v_+\, c_1\, p_+ ~ <~ \frac{\pi-\theta_1}{2}.
\end{equation}
(A similar argument shows that 
$\angle\, v_- c_3\, p_- < (\pi-\theta_1)/2$.)

We will use (\ref{est:cpv1}) to prove the following 
\begin{lem}
The minimal geodesic joining 
$c_3$ to $c_5$ is direct. 
\end{lem}

\begin{proof}
Let $\ell_1$ be the line 
parallel to $\overline{p_+p_-}$ that passes through $c_3$, 
and let $\ell_2$ be the line  
parallel to $\overline{v_-v_3}$ that passes through $v_+$. 
Since $\theta_2 < \pi/2$ and $|v_+v_2|=|v_2v_3|$, 
the points $v_2$ and $v_3$ 
lie in distinct components of $\Cbb - \ell_2$. 
Because $p_-$ lies in the component of $\Cbb -\overleftrightarrow{v_-v_3}$ 
that contains $v_2$ and $\overline{p_+p_-}$ is
a translate of $\overline{v_+v_-}$, the point $p_+$ 
lies in the component $H_2$ of $\Cbb - \ell_2$ that contains $v_2$. 
See Figure \ref{fig:direct-c3-c5}. 

\begin{figure}
\begin{center}
\includegraphics[width=10cm]{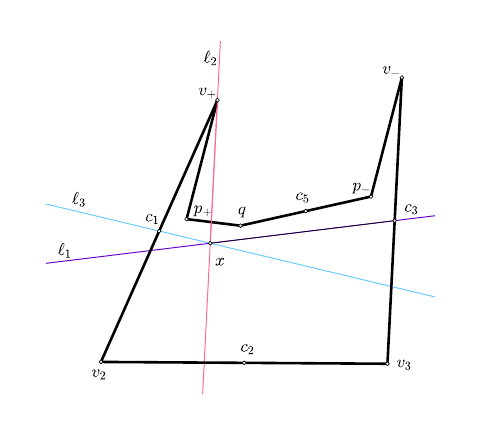}
\end{center}
\caption{The segment
that joins $c_3$ to $c_5$ belongs to $P$. 
\label{fig:direct-c3-c5}}
\end{figure}

Let $x$ be the point of intersection of
$\ell_1$ and $\ell_2$, and let $\ell_3$ be the line 
passing through $c_1$ and $x$. Since 
$|v_+x| = |v_-c_3|= \sys(X)/2= |v_+ c_1|$, 
the triangle $\triangle c_1 x v_+$ is isoceles.
Moreover, $\angle c_1 v_+ x = \theta_1$, and 
so $\angle v_+ c_1 x = (\pi -\theta_1)/2$. 
Therefore, if follows from (\ref{est:cpv1}) that 
$p_+$ lies in the component $H_3$ of $\Cbb- \ell_3$ 
that contains $v_+$.

Because $\theta_2 \leq \theta_3$ and 
$\theta_2+ \theta_3 < \pi$, the intersection 
$H_2 \cap H_3$ lies in the component $H_1$ of $\Cbb  -\ell_1$ 
that contains $v_+$. Thus, $p_+ \in H_1$ and since 
$\overline{p_+p_-}$ is parallel to $\ell_1$, we have
that $p_- \in H_1$. Hence, the angle 
$\angle c_3 p_-p_+$ is less than $\pi$. By Lemma \ref{lem:interior-angles},
the angle $\theta_-$ at $q$ is 
greater than $\pi$, and therefore we find that 
$\angle c_3 p_-q< \pi$. It follows that there is a direct
segment from $c_3$ to $c_5$ as desired. 
\end{proof}

\begin{lem}
The shortest geodesic that joins $c_1$ to $c_4$ is direct.
\end{lem}

\begin{proof}
Let $\ell_1$ be the 
line passing through $c_1$ that is parallel to $\overline{v_+v_-}$. 
Let $H_1$ be the the component $\Cbb-\ell_1$ that contains $v_+$.
It suffices to show that the point $p_+$ lies in  $H_1$. For then, since 
$\overline{p_+p_-}$ is parallel to $\ell_1$
and the angle $\theta_-< \pi$, it will follow that the
minimal geodesic joining $c_1$ to $c_4$ is direct. 

\begin{figure}
\begin{center}
\includegraphics[width=9cm]{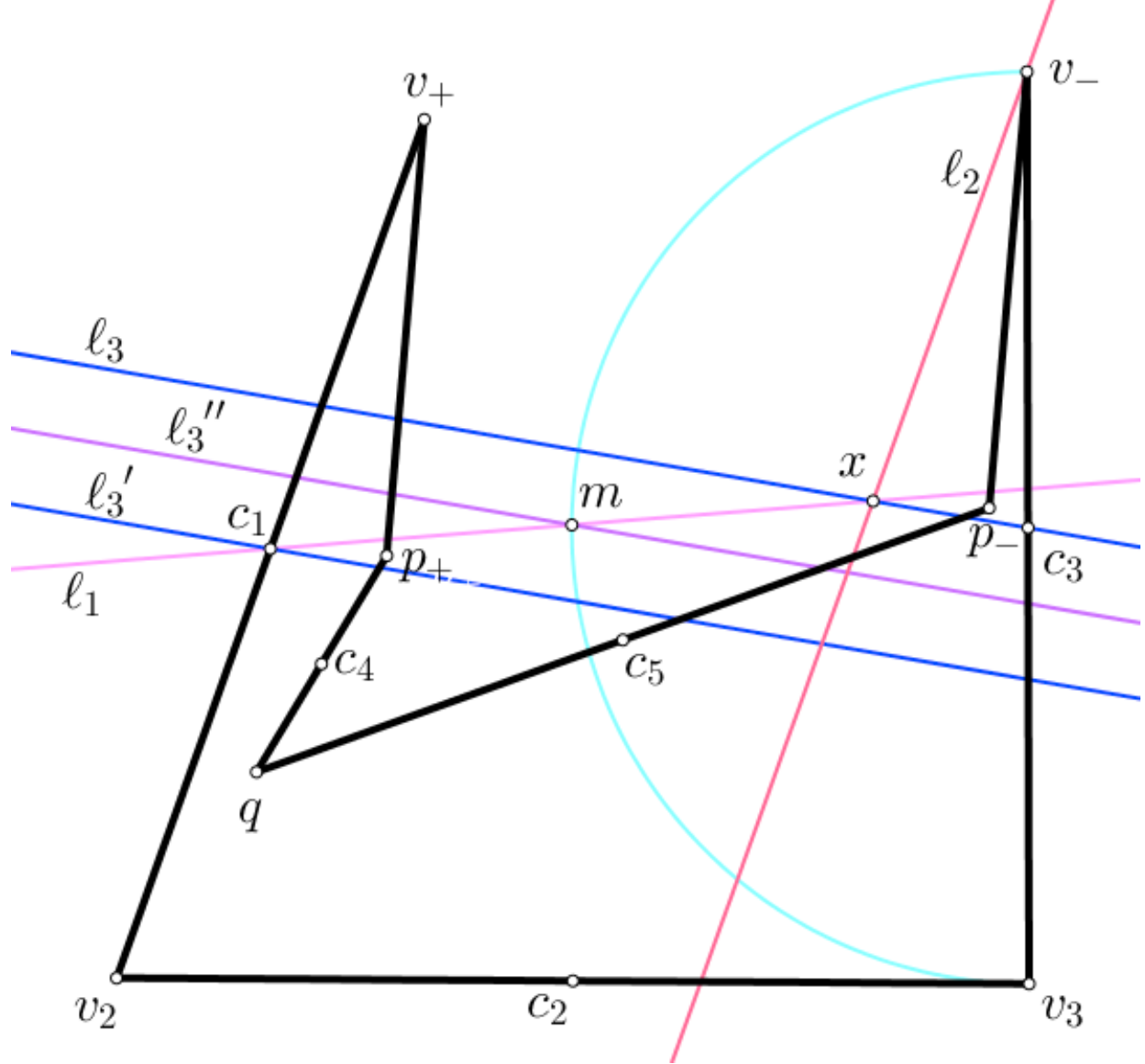}
\end{center}
\caption{The segment that joins $c_1$ to $c_4$ belongs to $P$. 
\label{fig:direct-c1-c4}}
\end{figure}

Suppose then, to the contrary, that $p_+$ belongs to $\Cbb-H_1$.
Then then since $\overline{p_-p_+}$ is parallel to 
$\ell_1 =\partial(\Cbb-H_1)$, the point $p_-$ also belongs to 
$\Cbb-H_1$. Moreover, since, by Lemma \ref{lem:interior-angles},
the angle $\theta_-$ at $q$ 
is larger than $\pi$, we also have $c_5 \in \Cbb-H_1$.
See Figure  \ref{fig:direct-c1-c4}.

Let $\ell_2$ be the line through $v_-$ that is parallel 
to $\overline{v_+p_+}$, and let $x$ be the intersection point 
of $\ell_1$ and $\ell_2$. Let $\ell_3$ be the line that passes
through $x$ and $c_3$. The triangle $\triangle x c_3v_-$
is isoceles, and in particular, $\angle xc_3v_-$ equals 
$(\pi-\theta)/2$. The argument analogous to that used to 
derive (\ref{est:cpv1}) gives the inequality 
$\angle p_- c_3 v_- < (\pi-\theta)/2$.
Therefore, $p_-$ lies in the component 
$H_3$ of $\Cbb - \ell_3$ that contains $v_-$. 

If we let $\ell_3'$ denote the line parallel to $\ell_3$ that 
passes through $c_1$, then, since $\overline{p_+p_-}$ is a 
translate of $\overline{c_1x}$, the point $p_+$ lies 
in the component $H_3'$ of $\Cbb - \ell_3'$ that contains $v_-$. 
Thus, to prove that there is a direct segment from $c_1$ 
to $c_4$, it suffices to show that $q$ lies in $\Cbb-H_3'$
for then $\angle c_1 p_+ c_4 < \pi$.

Let $m$ be the midpoint of $\overline{c_1x}$, and let 
$\ell_3''$ be the line parallel to $\ell_3$ that passes 
through $m$. To show that $q \in \Cbb-H_3'$, it suffices 
to show that $c_5$ lies in the closure of the 
component $H_3''$
of $\Cbb -\ell_3''$ that contains $v_2$. Indeed, 
$c_5$ is the midpoint of $\overline{p_-q}$ and 
we know that $p_-$ lies in $H_3$.

Since there is a direct segment joining $c_5$ to $c_3$,
the point $c_5$ lies outside the ball $B$ of radius $\sys(X)/2$
with center at $c_3$. We also know that $c_5$ lies in $Q$,
the convex hull of $\{v_+,v_2,v_3,v_-\}$, and 
that $c_5$ belongs to $\Cbb- H_1$. An elementary 
geometric argument shows that $(Q- B) \cap (\Cbb -H_1)$
lies in $\overline{H_3''}$. Thus, $c_5 \in \overline{H_3''}$
and there exists a direct segment joining $c_1$ to $c_4$ 
as desired. 
\end{proof}

Given that there are direct segments between $c_1$ and $c_4$
and between $c_3$ and $c_5$, we will now derive a contradiction and thus
complete the proof of Theorem \ref{prop:bivalent-direct} as follows. 

Let $\ell_+$ be the line that passes through $v_+$ and $c_4$,
let $\ell_-$ denote the line that passes through $v_-$ and $c_5$,
and let $x$ be the intersection of $\ell_+$ and $\ell_-$. 
See Figure \ref{fig:away}. Because $\overline{v_+p_+}$ and $
\overline{v_-p_-}$ are parallel and of the same length and the points
$c_4$, $c_5$ are the respective midpoints of $\overline{p_+q}$,
$\overline{p_-q}$, the points $c_4$, $c_5$ are also
the respective midpoints of $\overline{v_+x}$, $\overline{v_- x}$.
Since $c_1$ is the midpoint of $\overline{v_+v_2}$, 
we have $|xv_2|= 2 \cdot |c_1c_4|$. Since the geodesic 
from $c_1$ to $c_4$ is direct, we have $|c_1c_4| \geq \sys(X)/2$
and hence $|xv_2| \geq \sys(X)$. Similarly, since 
the geodesic from $c_3$ to $c_5$ is direct, we find 
that $|xv_3|\geq \sys(X)$.

\begin{figure}
\begin{center}
\includegraphics[width=9cm]{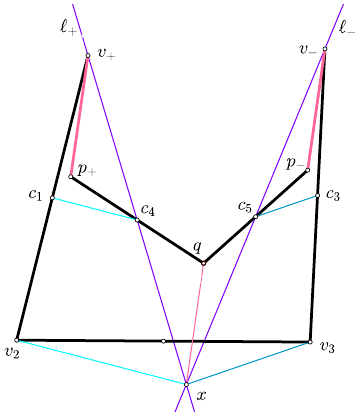}
\end{center}
\caption{The points $c_4$ and $c_5$ are the respective 
midpoints of $\overline{v_+x}$ and $\overline{v_-x}$. Thus, 
$|v_2x| = 2 |c_1c_4|$ and $|v_3x|= 2 |c_3c_5|$. \label{fig:away}}
\end{figure}

In other words, if we let $B_+$ (resp. $B_-$) be the ball of radius 
$\sys(X)$ about $v_2$ (resp. $v_3$), then $x$ lies outside 
$B_+ \cup B_-$. Since $\{v_+, v_2,v_3,v_-\}$ is contained 
$\oB_+ \cup \oB_-$, the polygon $P$ is contained 
in the convex hull of $\oB_+ \cup \oB_-$. 
 
Let $\ell_{23}$ denote the line passing through $v_2$ and $v_3$,
and let $y: \Cbb \to \Rbb$ denote the real affine 1-form such that
$|y(z)|=\dist(z,\ell_{23})$ and such that $y(v_+)>0$.
Because $\theta_2 \leq \theta_3 < \pi$, we have 
that $y(z) \geq 0$ for each $z \in P$. 

Note that $y(x) < y(q)$. Indeed, since 
$\angle c_1v_+p_+ < \theta_1$ and $\theta_1+\theta_2 < \pi$,
it follows that $y(v_+) > y(p_+)$. The segment $\overline{xq}$
is the reflection of $\overline{v_+p_+}$ about the point $c_4$,
and hence $y(x) < y(q)$. 

Let $x'$ be the intersection point of $\ell_{23}$ and the 
line passing through $x$ and $q$. The point $x'$ lies 
in the line segment $\overline{v_2v_3}$. 
Indeed, because $\theta_2 + \theta_3< \pi$, the
line through $v_+$ and $v_2$ and the line through
$v_-$ and $v_3$ intersect at a unique point $z$,
and moreover, the polygon $P$ lies in the convex
hull $T'$ of $\{z, v_2, v_3\}$. Because $\overline{p_-v_-}$
and $\overline{p_+v_+}$ are parallel, $p_+$ and $p_-$
lie in $T'$, $v_+$ lies in $\overline{zv_2}$, and
$v_-$ lies in $\overline{zv_3}$, 
any line parallel to ${p_+v_+}$ that intersects
$T'$ must intersect $\ell_{23}$ at a point in the segment
$\overline{v_2v_3}$. In particular, the point $x'$
lies in $\overline{v_2v_3}$.

We claim that $y(x) > 0$. Indeed, suppose not. Then $x'$ 
would lie in the segment $\overline{xq}$. Thus, 
$|x'x| \leq |xq|=|v_{\pm}p_{\pm}| < \sys(X)/2$, and 
hence $x$ would belong to the set, $A$, of points 
whose distance from $\overline{v_2v_3}$ is at most $\sys(X)/2$.
Elementary geometry shows that $A \subset B \cup B_+$, but 
$x$ lies in the complement of $B_- \cup B_+$, a contradiction. 

Let $Q$ be the convex hull of $\{v_+,v_2, v_3, v_-\}$. 
We have $P \subset Q$ and hence $q \in Q$. Since 
$0< y(x) < y(q)$ and the line through $x$ and $q$ 
meets $\ell_{23}= {\rm ker}(y)$ at $x' \in \overline{v_2v_3}$,
the point $x$ also belongs to $Q$. The set $Q$ is contained
in the convex hull of $\oB_+ \cup \oB_-$. Therefore, 
$x$ lies inside the 
convex hull of $\oB_+ \cup \oB_-$ and outside $B_+ \cup B_-$.
Since $x' \in \overline{v_2v_3}$ it follows that 
$\pi/4 \leq \angle v_2 x' x \leq 3\pi/4$, and, therefore, 
since $y(q) > y(x)$, we find that $q$ is also outside $B_+ \cup B_-$. 
See Figure \ref{fig:y-large}.

\begin{figure}
\begin{center}
\includegraphics[width=10cm]{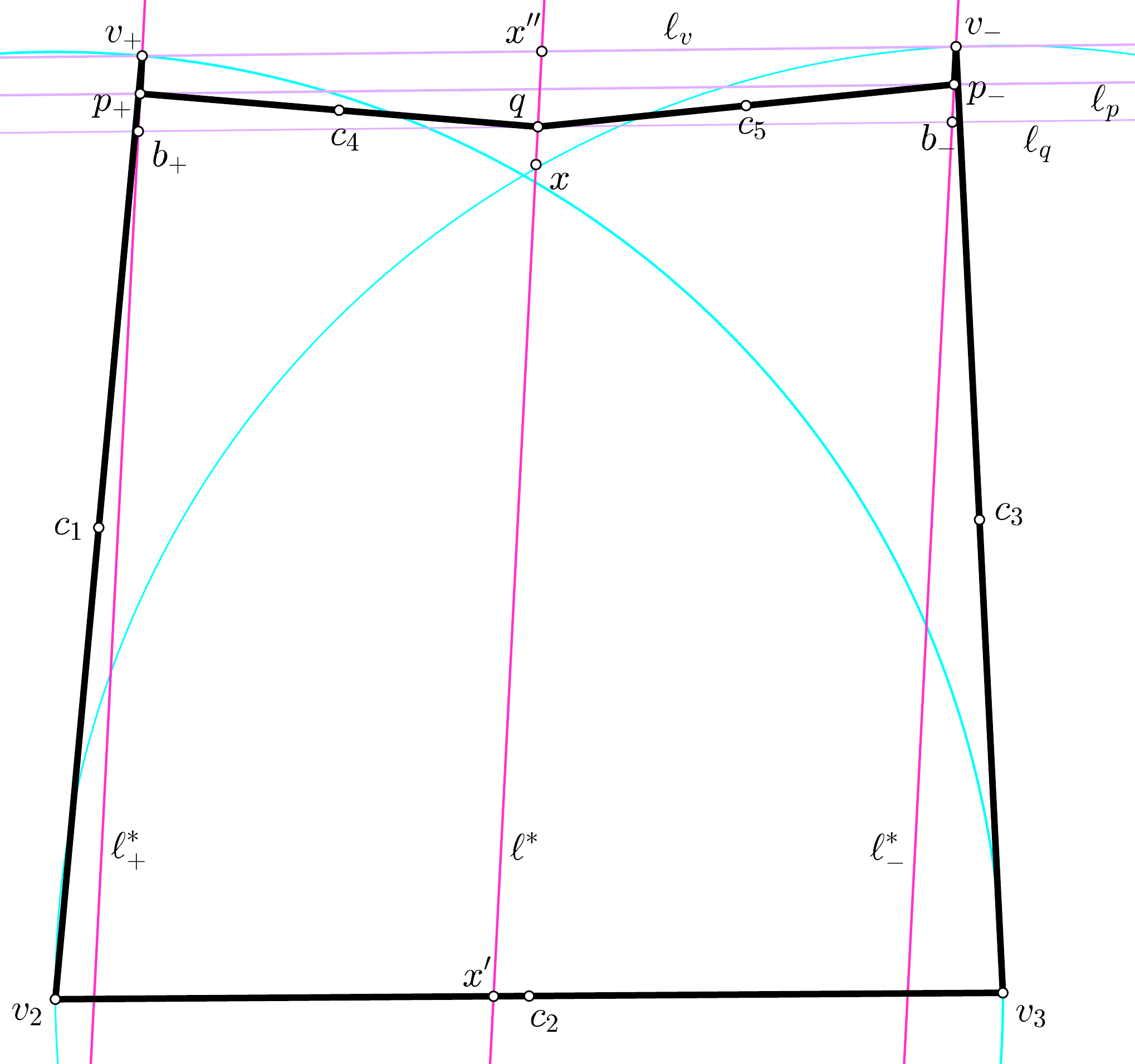}
\end{center}
\caption{The distances $|xv_2|$ and $|xv_3|$ are at least $\sys(X)$,
and $y(x)>0$. 
\label{fig:y-large}}
\end{figure}

Since $x$ and $q$ both lies inside the 
convex hull of $B_+ \cup B_-$ but outside $B_+ \cup B_-$,
we have $y(q)-y(x) < (1- \sqrt{3}/2) \cdot \sys(X)$. 
Since $\pi/4 \leq \angle v_2 x' x \leq 3\pi/4$, we have 
$|xq| \leq \sqrt{2} \cdot |y(q)-y(x)|$ and hence 
\begin{equation} \label{est:thin}
|v_{\pm}p_{\pm}|~ 
\leq~ 
\sqrt{2} \cdot \left(1- \frac{\sqrt{3}}{2} \right) \cdot \sys(X)~ 
<~ \frac{\sys(X)}{4}. 
\end{equation}

Let $\ell_p$ be the line through $p_+$ and $p_-$ and let $\ell_v$
be the line through $v_+$ and $v_-$. Let $\ell_{\pm}^*$ denote the 
line passing through $v_{\pm}$ and $p_{\pm}$.
By Lemma \ref{lem:interior-angles}, the interior angle $\theta_-$ at $q \in P$ 
is greater than $\pi$, and hence the point $q$ lies in the component of $\Cbb - \ell_p$ that 
contains the segment $\overline{v_2v_3}$, and hence $q$ 
lies in the component of $\Cbb -\ell_v$ that contains $\overline{v_2v_3}$.
Since $q$ lies outside $B_+ \cup B_-$, it follows that $q$ lies 
in the bounded component of 
$\Cbb - (\ell_{23} \cup \ell_{v} \cup \ell_+^* \cup \ell_-^* )$. 

Let $\ell_q$ be the line through $q$ that it parallel to $\ell_v$. 
Let $A$ be the parallelogram that is the bounded component
of $\Cbb - (\ell_{q} \cup \ell_{v} \cup \ell_+ \cup \ell_- )$.
Let $b_{\pm}$ be the intersection of $\ell_{\pm}$ and $\ell_q$.
Then $A$ is the convex hull of $\{b_+, b_-, v_+, v_-\}$. 
Because $q$ lies in the component of $\Cbb - \ell_p$ that 
contains $\overline{x_2x_3}$, the point $p_{\pm}$ lies in 
$\overline{v_{\pm} b_{\pm}}$. 

The line $\ell^*$ through $x$ and $q$ is parallel to the sides
corresponding to $\ell_+$ and $\ell_-$. Let $x''$ be the intersection
of $\ell^*$ with the side $\overline{v_+v_-}$ of $A$ corresponding to $\ell_v$. 
Since $v_{\pm} \in B_+ \cup B_-$, the point $x''$ lies in 
the convex hull of $B_+\cup B_-$. By applying the argument that
led to (\ref{est:thin}) to this situation, we find 
that $|x''q| < \sys(X)/4$. 

We have $|b_+b_-|= |v_+v_-| < \sys(X)$ and hence either 
$|b_+q|< \sys(X)/2$ or $|b_-q|< \sys(X)/2$.
Suppose that $|b_+q|< \sys(X)/4$. The midpoint, $c_4$,
of $\overline{p_+q}$
lies in $A$. Let $a_+$ be the point of intersection of
$\ell_+$ and the line through $c_4$ that is parallel to $\ell_q$.
Then $a_+$ lies in the segment $\overline{p_+b_+}$. 

By the triangle inequality, we have
\[ |v_{+} p_+| + |p_+ c_4|~ 
\leq~ |v_{+}a_+|~ +~ |a_+c_4|~
<~ \frac{\sys(X)}{4} + \frac{\sys(X)}{4}~ 
=~ \frac{\sys(X)}{2} 
\]
But $v_+$ and $c_4$ are both Weierstrass points, and hence
we would have a curve of length less than $\sys(X)/2$.
A similar contradiction is obtained in the case when 
$|b_-q|< \sys(X)/2$.
\end{proof}

The following is immediate.

\begin{coro}
There are at most six homotopy classes of nonseparating systoles.
\end{coro}

%%%%%%%%%%%%%%%%%%%%%%%%%%%%%%%%%%%%%%%%%%%%%%%%%%%%%%%

\section{Indirect Weierstrass arcs} \label{sec:indirect}

The angle $4 \pi$ cone point $c^*$ divides each systolic indirect Weierstrass arc
on $X/ \langle \tau\rangle$ into two subarcs. We will call each such subarc a {\it prong}.

Let $C_{\epsilon}$ be the set of points at distance $\epsilon$ from $c^*$. 
For $\epsilon$ sufficiently small, the set $C_{\epsilon}$ 
is a topological circle, and each prong intersects $C_{\epsilon}$
exactly once. Thus, the prongs divide the circle $C_{\epsilon}$ into disjoint arcs.
Two prongs are said to be {\it adjacent} if they are joined by one of these
arcs, and the {\it angle} between two adjacent arcs is the arclength
divided by $\epsilon$.

If a systolic indirect Weierstrass arc is the union of two adjacent prongs
then the angle between the two prongs must be at least $\pi$.
Indeed, otherwise one can shorten the arc by perturbing it near $c^*$.

The sum of the lengths of any two prongs is at least $\sys(X)/2$. 
Indeed, otherwise the concatenation of the two prongs would lift  
to a geodesic loop on $X$ that would have length less than $\sys(X)$.
On the other hand, for each prong, there is another prong so that
the sum of the lengths of the two prongs equals $\sys(X)/2$. 
 
In particular, the minimum, $\ell$, of the lengths of the prongs is at most $\sys(X)/4$.
If $\ell < \sys(X)/4$, then there is a unique shortest prong and the 
remaining prongs have length $\sys(X)/2-\ell$. 

If $\ell = \sys(X)/4$, then each prong has length $\sys(X)/4$,
and each pair of adjacent prongs determines a systolic Weierstrass arc.
Since the angle between each adjacent pair is at least $\pi$ and $c^*$
has total angle $4 \pi$, there are at most four adjacent pairs
and if there are exactly four pairs, then each angle equals $\pi$.
In sum, we have

\begin{prop} \label{prop:equal-four-prongs}
If all of the prongs have the same length, then the number
of prongs is at most four. If there are exactly four such prongs, then 
the angle between each pair of adjacent prongs is exactly $\pi$. 
\end{prop}

We will show below that if one of the prongs is shorter than the others
then there are at most five prongs. To do this we will use the following lemma.

\begin{lem} \label{lem:distinct-end-prongs}
Two distinct prongs can not end at the same angle $\pi$ cone point, $c'$.
\end{lem}

\begin{proof}
Suppose not. Then the concatenation, $\alpha$, of the two prongs would be
a closed curve that divides the sphere $X/\langle \tau \rangle$ into two discs.
Since there are five other cone points, one of the discs, $D$, would contain
at most two cone points. There are no Euclidean bigons and so $D$ would
have to contain at least one cone point.

If $D$ were to contain two angle $\pi$ cone points,
then $\alpha$ would be homotopic to the concatenation of the two oriented
minimal arcs joining the two cone points. The length of the
unoriented minimal arc is at least $\sys(X)/2$, and hence, since
the length of each prong is less than $\sys(X)/2$, we would have a contradiction.

If $D$ were to contain one angle $\pi$ cone point $c$, then $\alpha$
would be homotopic to the concatenation of the two oriented
minimal arcs joining $c$ and $c'$. We would then arrive at a contradiction as
in the case of two cone points.
\end{proof}

Since there are exactly six Weierstrass points,
Lemma \ref{lem:distinct-end-prongs} implies that there
are at most six prongs. In fact,
we have the following.

\begin{prop} \label{prop:at-most-five-prongs}
There are at most five prongs. 
\end{prop}

\begin{proof}
Suppose to the contrary that there are six prongs. Let $e_1$
denote the unique shortest prong, let $\ell$ be its length,
and let $c_1$ denote its endpoint. Let $e_1, \ldots, e_6$ be a cyclic
ordering of the remaining prongs, let $L=\sys(X)/2-\ell$ denote
their common length, and let $c_2, \ldots, c_6$ denote their respective endpoints.

Since $\ell(e_1+e_2)= \sys(X)/2= \ell(e_1+e_6)$, the
angles $\angle c_1 c^* c_2$ and $\angle c_1 c^* c_2$ are each at least $\pi$.
(Otherwise, by perturbation near the $4 \pi$ cone point we could
construct a direct Weierstrass arc with length less than $\sys(X)/2$.)
Each of the other four angles between adjacent prongs is greater than $\pi/3$.
Indeed, otherwise, since $L < \sys(X)/2$, we would have a segment joining
two angle $\pi$ cone points having length less than $\sys(X)/2$ which contradicts
the definition of systole. Since $\angle c_1 c^* c_2 + \angle c_1 c^* c_6 \geq 2\pi$
it follows that each of these four angles is less than $\pi$. 
Moreover, since the angle at $c^*$ equals $4\pi$,
the sum $\angle c_1 c^* c_2 + \angle c_1 c^* c_6 < 8\pi/3< 3 \pi$
and individually $\angle c_1 c^* c_2 < 5\pi/3$ and $\angle c_1 c^* c_6 < 5\pi/3$.

By cutting along the prongs
and taking the length space completion, we obtain a closed topological
disc $D$ whose boundary consists of a topological disc bounded by six geodesic
segments. The midpoint of each segment corresponds to an end point of a prong.
The developing map provides an immersion of $D$ into the Euclidean plane.
Since $\angle c_k c^* c_{k+1} + \angle c_1 c^* c_2 < 3 \pi$ and
$\angle c_i c^* c_{i+1} < \pi$ for $i=2,\ldots, 5$, this immersion is
an embedding. In other words, we may regard $D$ as Euclidean hexagon.
 
Let $v_i$ denote the vertex of $D$ corresponding to $c^*$ that lies between
$c_{i-1}$ and $c_{i}$. The length of the side $\overline{v_1v_2}$ is $2\ell$, and the
common length of the other sides is $2 L$.
From above, the interior angles at $v_1$ and $v_6$ are between $\pi$ 
and $5\pi/3$, and the angles at the other four vertices lie between $\pi/3$ and $\pi$.
Without loss of generality, $c_1 =(0,0)$, $v_1=(\ell,0)$, $v_6=(-\ell,0)$ and
an $H$-neighborhood of $c_1$ lies in the upper half plane 
(see Figure \ref{Fig-at-most-5}).

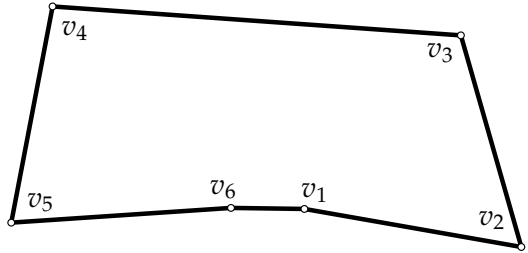
\begin{figure}[h]\label{Fig-at-most-5}
%\ShowGrid
%{\color{linkred}
\leavevmode \SetLabels
%\L(.39*.37) $\sqrt{3}$\\%
\L(.53*.235) $v_1$\\
\L(.685*.13) $v_2$\\%
\L(.64*.775) $v_3$\\%
\L(.32*,86) $v_4$\\
\L(.29*.20) $v_5$\\%
\L(.45*.25) $v_6$\\%
\endSetLabels
\begin{center}
\AffixLabels{\centerline{\epsfig{file =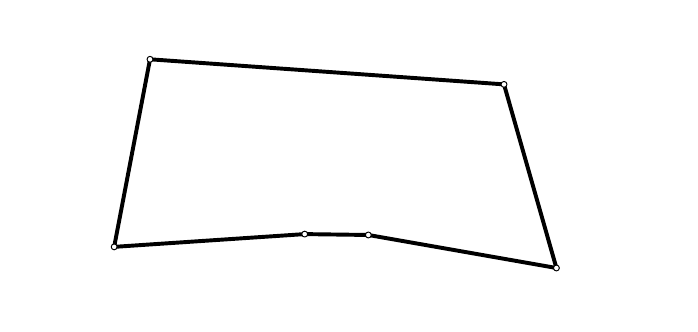,width=7cm,angle=0}}}
\vspace{-24pt}
\end{center}
\caption{The points $v_i$, $i=1,\hdots,6$.  \label{Fig-at-most-5}}
\end{figure}

Since the angle at $v_2$ (resp. $v_5$) is greater than $\pi/3$, and the edges
$\overline{v_1v_2}$ and $\overline{v_2v_3}$
(resp.$\overline{v_4v_5}$ and $\overline{v_5v_6}$) have length $2 L$,
the vertex $v_3$ (resp. $v_4$) lies outside the ball of radius $2L$
centered at $v_1$ (resp. $v_6$). It follows that if
both $v_4$ and $v_3$ both lie in the lower half plane then
the shortest arc in $H$ that joins $v_4$ to $v_3$ has distance at
least $2L+ 2 \ell$. This contradicts the equality $|v_3v_4|= 2 L$.

Since the angle at $v_1$
(resp. $v_6$) is at least $\pi$ and the angle at $v_2$ (resp. $v_5$) is
greater than $\pi/3$, if $v_3$ (resp. $v_4$) lies in the upper half plane,
then $v_3$ (resp. $v_4$) lies in the half plane $V_+=\{(x_1,x_2)~ |~ x_1 > \ell+L\}$
(resp. $V_-=\{(x_1,x_2)~ |~ x_1 < -\ell-L\}$).
Since the distance between $U_+$ and $U_-$ equals $2L+2 \ell$,
if $v_3$ and $v_4$ both lie in $U$, then we contradict $|v_3v_4|= 2 L$.

If $v_3$ lies in the upper half plane and
that $v_4$ lies in the lower half plane but not in $U_-$,
then $v_4$ lies in the half plane that is bounded by the line trough $v_3$ and $v_6$
and contains $v_1$. In particular, the shortest path in $D$ between $v_3$ and $v_4$
passes through $v_6$. But the distance from $v_6$ to $U_+$ is equal to $2L + \ell$,
and the distance form $v_6$ to $v_4$ is greater than $2L$. Thus, we contradict
$|v_3v_4|= 2 L$. 

A symmetric argument rules out the remaining case in which
the r\^oles of $v_3$ and $v_4$ are reversed.
\end{proof}

\begin{thm} \label{thm:at-most-six-indirect}
There are at most six systolic indirect Weierstrass arcs.
Equality occurs if and only if there are exactly four prongs 
and these four prongs have the same length. 
\end{thm}

\begin{proof}
If the prongs are not all of the same length,
then one prong has length less than $\sys(X)/4$
and hence the others have length greater than $\sys(X)/4$. Therefore,
concatenations of none of the others constitute a systolic Weierstrass arcs.
By Proposition \ref{prop:at-most-five-prongs}, there are at most five prongs
and hence at most five systolic indirect Weierstrass arcs.

If the prongs all have the same length---namely $\sys(X)/4$---then by Proposition
\ref{prop:equal-four-prongs} there are at $n \leq 4$ prongs.
Each concatenation of a pair prongs constitutes a systolic Weierstrass arc,
and so there are exactly $n \cdot (n-1)/2$ prongs and hence at most six.
Six occurs if and only if $n=4$. 
\end{proof}

%%%%%%%%%%%%%%%%%%%%%%%%%%%%%%%%%%%%%%%%%%%%%%%%%%%%%%%%

\section{A separating systole}

In this section we wish to prove the following:

\begin{thm} \label{thm:separating-at-most-nine}
If $X$ has a separating systole $\alpha$, then $X$ has at most nine homotopy classes
of closed curves with systolic representatives.
\end{thm}

We will use the argument explained in Remark \ref{remk:surgery} in the three lemmas that follow. We first observe:

\begin{lem} \label{lem:at-most-one-sep}
$X$ has at most one separating systole.
\end{lem}
\begin{proof}
Suppose there are two separating systoles. Each angle of intersection between the two curves 
must be at least $\pi$, otherwise one can find a shorter non-homotopically trivial curve by 
a cut-and-paste argument. Hence intersection points between the systoles occur at the $4\pi$ 
cone points. But as any two separating curves intersect at least $4$ times, this is impossible 
because there are only two angle $4\pi$ cone points.
\end{proof}

\begin{lem} \label{lem:sep-and-direct-n-intersect}
If $\alpha$ is a separating systole and $\gamma$ is a direct systolic Weierstrass arc,
then $\gamma$ does not intersect the projection of $\alpha$ to $X/ \langle \tau \rangle$.
\end{lem}
\begin{proof}
Suppose not. The lift, $\tgamma$, of $\gamma$ to $X$ is a systole that does
not pass through an angle $4\pi$ cone point.
Since $\alpha$ is separating, the curve $\tgamma$ intesersects $\alpha$ at least
twice. Let $p_-$ and $p_+$ be two of the intersection points. The points
$p_+$ and $p_-$ divides $\alpha$ (resp. $\tgamma$) into a pair
of arcs. One of the arcs, $\alpha_-$ (resp. $\tgamma_-$), has length at
most $\sys(X)/2$. By concatenating $\alpha_-$ and $\tgamma_-$, we obtain a
non null homotopic closed curve $\beta$ of length at most $\sys(X)$. Since each
intersection point is not a cone point and the geodesics are distinct,
the angle at each intersection point $\tgamma_-$ is less than $\pi$. Thus, a perturbation of $\beta$ near an intersection point produces a curve homotopic to $\beta$ that has shorter length, a contradiction.
\end{proof}

\begin{lem} \label{lem:sep-no-short-prong}
If $X$ has a separating systole $\alpha$, then each prong of $X$ 
has length equal to $\sys(X)/4$.
Moreover, the angle between the projection $p(\alpha)$ 
and each prong is at least $\pi$.
\end{lem}

\begin{proof}
If not, then by the discussion at the beginning of \S \ref{sec:indirect},
there would exist a prong of length strictly less
than $\sys(X)/4$. The preimage of a prong under $p$ is an arc $\gamma$ of
length $\sys(X)/2$ that joins one angle $4 \pi$ cone point $c_-^*$ to the
other angle $4\pi$ cone point $c_+^*$.
By Corollary \ref{coro:slit}, the separating systole $\alpha$ passes through
both $c_-^*$ and $c_+^*$, and the complement $\alpha \setminus \{c_-^*, c_+^*\}$
consists of two arcs $\alpha_+$ and $\alpha_-$ each of length $\sys(X)/2$.
By concatenating $\alpha_{\pm}$ with $\gamma$ we would obtain a non-null homotopic
closed curve having length less than $\sys(X)$, a contradiction.

If the angle between the prong and $p(\alpha)$ were less than $\pi$, then
one could perturb the concatenation of $\alpha_{\pm}$ and $\gamma$ to obtain
a non-null homotopic closed curve whose length
would be less than $\sys(X)/2$, a contradiction.
\end{proof}

\begin{proof}[Proof of Theorem \ref{thm:separating-at-most-nine}]
Let $\alpha$ denote the separating systole to $X/ \langle \tau\rangle$
which is unique by Lemma \ref{lem:at-most-one-sep}. 
By Lemma \ref{lem:sep-no-short-prong}, each prong has length equal
to $\sys(X)/4$ and the angle between $p(\alpha)$ and each prong is at least $\pi$.
Thus, since the total angle at $c^*$ is $4\pi$, there are at most two prongs.
Hence there are at most two indirect systolic Weierstrass arcs.

By Theorem \ref{prop:bivalent-direct}, there are at most six direct systolic
Weierstrass arcs. Thus, by Proposition \ref{prop:nonseparating} and the
discussion at the beginning of \S \ref{sec:direct-arcs}, there are at most
eight homotopy classes of non-separating closed curves that have systolic
representatives. Since $\alpha$ is the unique separating systole, the claim
is proven.
\end{proof}

As a corollary of the proof of Theorem \ref{thm:separating-at-most-nine}
 and Lemma \ref{lem:sep-no-short-prong}, we have the following.

\begin{coro} \label{coro:separating-at-most-two-prongs}
If $X$ has a separating systole, then $X$ has either no prongs 
or exactly two prongs of the same length. 
\end{coro}

%%%%%%%%%%%%%%%%%%%%%%%%%%%%%%%%%%%%%%%%%%%%%%%%%%%%%%%%%

\section{Crossing systoles}

In this section we prove the following:

\begin{thm}
Suppose that $X/\langle \tau \rangle$ has exactly four prongs and each of these 
prongs has length equal to $\sys(X)/4$.
Then at most ten homotopy classes of closed curves are represented by systoles. 
Moreover, if $X$ has exactly ten homotopy classes of systoles, then $X$ is homothetic to
the surface described in Figure \ref{X10Figure}, and otherwise
$X$ has at most eight homotopy classes of systoles.
\end{thm}

\begin{proof}
By Corollary \ref{coro:separating-at-most-two-prongs}, the surface $X$ has no separating systole. 
By Theorem \ref{thm:at-most-six-indirect}, there are exactly six indirect
systolic Weierstrass arcs. Thus, by Proposition \ref{prop:nonseparating} and the
discussion at the beginning of \S \ref{sec:direct-arcs}, to prove the first claim
it suffices to show that there are at most four direct systolic Weierstrass arcs.
 
By Proposition \ref{prop:equal-four-prongs}, the angle between adjacent
prongs equals $\pi$. Thus, by cutting along the four prongs we obtain a
topological disc $D$ bounded by a geodesic $\beta$ with no corners. The
geodesic $\beta$ has length $8 \cdot (\sys(X)/4) = 2 \cdot \sys(X)$
and contains one point corresponding to each of the angle $\pi$
cone points that are endpoints of the four prongs.
Label those cone points in cyclic order $c_1, c_2, c_3$, and $c_4$.
For each $i$, there is a unique point $c_i^*$ on $\beta$ lying between $c_i$ and $c_{i+1}$
that corresponds to $c^*$.
The distances satisfy $\dist(c_i,c_i^*)= \sys(X)/4 = \dist(c_i^*, c_{i+1})$.
The interior angle at each $c_i$, $c_i^*$ is $\pi$. 

The two remaining angle $\pi$ cone points, $c_5$ and $c_6$, lie in the interior
of the disc $D$. Because $\beta$ is a geodesic (without corners), the disk
is geodesically convex, and there exists a direct Weierstrass arc $\gamma$
joining $c_5$ and $c_6$. By cutting along $\gamma$ we obtain a topological
annulus $A$ with geodesic boundary components $\beta$ and $\beta'$.
Since $X$ is a translation surface, $A$ is a Euclidean cylinder isometric
to $[0, h] \times (\Rbb/\ell \cdot \Zbb)$ where $\ell= 2 \cdot \sys(X)$ is the
common length of $\beta$ and $\beta'$.

The length of $\gamma$ equals $(1/2)\cdot \ell$, and hence $\gamma$ is not systolic.
The distance between $c_5$ (resp. $c_6$) and $\{c_1, c_2,c_3, c_4\}$ is at
least $\sys(X)/2$. It follows that the height $h$ of the cylinder $A$ is at
least $(\sqrt{3}/4) \cdot \sys(X)$. As a consequence, there does not exist
a direct systolic Weierstrass arc joining two distinct points in
$\{c_1, c_2,c_3, c_4\}$.

In sum, if $\delta$ is a direct systolic Weierstrass arc, then $\delta$
joins a point in $\{c_5,c_6\}$ to a point in $\{c_1, c_2,c_3, c_4\}$.
Since $A$ is a Euclidean annulus, there are at most two direct systolic
Weierstrass arcs joining $c_5$ (resp. $c_6$) to $\{c_1, c_2,c_3, c_4\}$,
and hence at most ten systolic Weierstrass arcs in total.
 
Moreover, since the points $\{c_1, c_2,c_3, c_4\}$ are evenly spaced
around $\beta$, and the points $\{c_5,c_6\}$ are evenly spaced about $\beta'$,
there are exactly four systolic arcs only if the respective
shortest segments, $\sigma_5$ and $\sigma_6$,
joining $c_5$ and $c_6$ to $\beta$ bisect arcs joining successive
points in $\{c_1, c_2,c_3, c_4\}$, that is, only if $\sigma_5$
and $\sigma_6$ have endpoints in $\{c_1^*, c_2^*,c_3^*, c_4^*\}$.
In this case, $h = (\sqrt{3}/4) \cdot \sys(X)$. It follows that
$X$ is homothetic to the surface described in Figure \ref{X10Figure}.

Finally, if there is only one direct systolic Weierstrass arc joining $c_5$ (resp. $c_6$)
to $\{c_1, c_2,c_3, c_4\}$, then there is only one direct systolic Weierstrass arc
joining $c_6$ (resp. $c_5$). Hence, if $X$ is not homothetic to the surface
described in Figure \ref{X10Figure}, then $X$ has at most eight homotopy
classes of simple closed curves with systolic representatives.
\end{proof}

%%%%%%%%%%%%%%%%%%%%%%%%%%%%%%%%%%%%%%%%%%%%%%%%%%%%%%%%%%

\section{One short prong}

In this section we prove the following:

\begin{thm}
If $X/\langle \tau\rangle$ has one short prong, then $X$ has at most nine homotopy classes
of closed curves that are represented by systoles.
\end{thm}

\begin{proof}
By Corollary \ref{coro:separating-at-most-two-prongs}, the surface $X$ has no separating systole.
By Proposition \ref{prop:at-most-five-prongs}, there are at most five prongs,
and so by assumption there is one prong of length $\ell < \sys(X)/4$
and four prongs of length $L=\sys(X)/2-\ell$. Thus, there are at most four
indirect systolic Weierstrass arcs. Thus, it suffices to show that
$X$ has at most five direct systolic Weierstrass arcs.

By cutting $X/\langle \tau\rangle$ along the five prongs, we obtain a topological
disc $D$ with one angle $\pi$ cone point in the interior. The boundary consists of
five geodesic arcs each of whose endpoints---vertices---corresponds to the angle $4\pi$ cone point.
The midpoint of each arc corresponds to an angle $\pi$ cone point on $X/\langle \tau\rangle$.
Choose an orientation of the boundary, and let $c_1^*$ and $c_2^*$ denote the endpoints of
the oriented arc that corresponding to the short prong.
Label the other vertices $c_3^*$, $c_4^*$, and $c_5^*$ according to the orientation.
Denote by $c_i$ the midpoint of the arc with endpoints $c_i^*$ and $c_{i+1}^*$.
There remains one angle $\pi$ cone point, $c_6$, that belongs to the interior of $D$.

By Theorem \ref{prop:bivalent-direct}, for each angle $\pi$ cone point $c_i$,
there are at most two direct systolic Weierstrass arcs ending at $c_i$.
Thus, to prove the claim, it suffices to show that $c_1$ is the endpoint
of at most one direct systolic Weierstrass arc. We will show
that if $c_1$ is the endpoint of a direct systolic Weierstrass arc,
then the other endpoint must be $c_6$.
 
Since systolic Weirstrass arcs can not intersect except at a cone point,
a direct Weierstrass arc joining $c_1$ to another angle $\pi$ cone point
can not pass through the boundary of $D$. In particular, if $\alpha$ is a
direct Weierstrass arc joining $c_1$ to either $c_2$, $c_3$, $c_4$, or $c_5$,
then the complement of $\alpha$ consists of two disks, one that contains
$c_6$ and one that does not.

Suppose that $\alpha$ is a direct geodesic segment that joins $c_1$ and $c_2$.
Consider the component, $D'$, of $D \setminus \alpha$, containing $c_1^*$.
If $D'$ does not contains $c_6$, then $D'$ is a flat surface bounded by
three geodesic segments. Since the angle at $c_1^*$ is at least $\pi$,
the Gauss-Bonnet formula implies that the angles at $c_1$ and $c_2$
are both zero, and hence $\alpha$ is not direct. 

If $D'$ contains $c_6$, then by cutting $D'$ along the geodesic segment
joining $c_6$ and $c_1^*$
we obtain a quadrilateral $Q$ with a side corresponding to $\alpha$.
The endpoints of $\alpha$ correspond to $c_1$ and $c_2$. Let $x_-$ and $x_+$
denote the vertices of $Q$ distinguished by $|x_-c_1|=\ell$ and $|x_+c_2|=L$.
If $\alpha$ is systolic, then, by the triangle inequality,
$|c_1x_+| \leq L + \sys(X)/2$ with equality if and only if $c_1$, $c_2$ and $x_+$ are colinear.
The midpoint of $\overline{x_- x_+}$ is $c_6$, and thus
by the triangle inequality
\[ |c_1 c_6|~ \leq~ \frac{|c_1x_-|}{2}~ +~ \frac{|c_1x_+|}{2}~ \leq~ \ell+ L~
 =~ \frac{\sys(X)}{2}
\]
with equality $c_1$, $c_6$, and $c_2$ are colinear.
Thus, either $|c_1c_6|< \sys(X)/2$ or $|c_2c_6|< \sys(X)/2$, a contradiction.
Therefore, there is no direct systolic Weierstrass arc joining $c_1$ and $c_2$.
Similarly, there is no direct systolic Weierstrass arc joining $c_1$ and $c_5$.

Suppose that $\alpha$ is a direct geodesic segment that joins $c_1$ to $c_3$.
Let $D'$ denote the component of $X \setminus \alpha$ that contains $c_2$.
If $D'$ does not contain $c_6$, then $D'$ is a quadrilateral with vertices
$c_1$, $c_1^*$, $c_2^*$, and $c_3$. Since $|c_2c_2^*|=L=|c_2^*c_3|$,
the angle $\angle c_3 c_2c_2^*$ is less than $\pi/2$, and thus
$\angle c_1^*c_2c_3 > \pi/2$. Therefore $|c_1^*c_3|> |c_2c_3| \geq \sys(X)/2$. 
Because $|c_1^*c_2^*|=2L$ and $|c_2^*c_3|=L$, the angle $\langle c_2 c_1^* c_3$
is acute. Thus, since the interior angle at $c_1^*$ is at least $\pi$,
the angle $\angle c_1 c_1^* c_3$ greater than $\pi$. In particular,
 $|c_1c_3|> |c_1^*c_3|$, and so, in sum, the length of $\alpha$
is greater than $\sys(X)/2$.

If $D'$ contains $c_6$, then the other component of $D \setminus \alpha$,
is a pentagon with vertices $c_1$, $c_3$, $c_3^*$, $c_4^*$, and $c_5^*$.
Using the triangle inequality, we have
\[ L+ |c_3 c_5^*|\, \geq\, |c_3c_3^*|+ |c_3^*c_5^*|\, \geq\,
 |c_3^*c_5^*|\, =\, 2 |c_4c_5|\, \geq\, \sys(X)\, =\, 2 \ell + 2 L,
\]
and therefore $|c_3c_5^*| \geq 2 \ell+L > \ell+L = \sys(X)/2$.

Since $|c_3^*c_5^*| \geq \sys(X) > 2L = |c_3^*c_4^*|= |c_4^*c_3^*|$,
the angle $\angle c_3^*c_5^*c_3$ is less than $\pi/3$.
Because $|c_3^*c_5^*| > 2L = |c_3^*c_3|$, we have
$\angle c_3^*c_5^* c_3 < \pi/6$. Thus, since the interior angle
at $c_5^*$ is at least $\pi$, the angle $\angle c_1c_5^*c_3$ is
greater than $\pi/2$. Therefore, $|c_1c_3| > |c_3c_5^*|$.
In sum, $|c_1c_3| > \sys(X)/2$, and hence $\alpha$ is not systolic.
Therefore, there is no direct systolic Weierstrass arc joining $c_1$
to $c_3$. A similar argument shows that there is no direct
systolic Weierstrass arc joining $c_1$
to $c_4$. 
\end{proof}

%%%%%%%%%%%%%%%%%%%%%%%%%%%%%%%%%%%%%%%%%%%%%%%%%%%%%%%%%%%%%%%%%

{\it Addresses:}\\
Department of Mathematics, Indiana University, Bloomington, IN, USA\\
Mathematics Research Unit, University of Luxembourg, Esch-sur-Alzette, Luxembourg

{\it Emails:}\\
cjudge2@gmail.com\\
hugo.parlier@uni.lu

\end{document}